                    \def\version{July 29, 2016}                       %

 \documentclass[reqno,11pt]{amsart}
 \usepackage{amsmath, amsthm, a4, latexsym, amssymb}
\usepackage[unicode]{hyperref}
\usepackage{srcltx}
\usepackage{xcolor}

\def\@rmrk#1#2{\refstepcounter
    {#1}\@ifnextchar[{\@yrmrk{#1}{#2}}{\@xrmrk{#1}{#2}}}

\makeatletter\@addtoreset{equation}{section}\makeatother

 \newfont{\bfit}{cmbxti10 scaled 1200}

\renewcommand{\d}{{\rm d}}
 \newcommand{\e}{{\rm e} }

 \newcommand{\eps}{\varepsilon}

 \newcommand{\R}{\mathbb{R}}
 \newcommand{\N}{\mathbb{N}}

 \newcommand{\E}{\mathbb{E}}
 \renewcommand{\P}{\mathbb{P}}
 \def\1{{\mathchoice {1\mskip-4mu\mathrm l} 
{1\mskip-4mu\mathrm l}
{1\mskip-4.5mu\mathrm l} {1\mskip-5mu\mathrm l}}}

 \newcommand{\Mcal}{{\mathcal M}}

 \newcommand{\weak}{{\Rightarrow}}

\newcommand{\heap}[2]{\genfrac{}{}{0pt}{}{#1}{#2}}

\newcommand{\ssup}[1] {{\scriptscriptstyle{({#1}})}}

\newenvironment{Proof}[1]
{\vskip0.1cm\noindent{\bf #1}}{\vspace{0.15cm}}
\renewcommand{\subsection}{\secdef \subsct\sbsect}
\newcommand{\subsct}[2][default]{\refstepcounter{subsection}
\vspace{0.15cm}
{\flushleft\bf \arabic{section}.\arabic{subsection}~\bf #1  }
\nopagebreak\nopagebreak}
\newcommand{\sbsect}[1]{\vspace{0.1cm}\noindent
{\bf #1}\vspace{0.1cm}}

{\nopagebreak {\hfill\rule{2mm}{2mm}}\\ }

\newtheorem{theorem}{Theorem}[section]
\newtheorem{lemma}[theorem]{Lemma}
\newtheorem{cor}[theorem]{Corollary}

\newtheoremstyle{thm}{1.5ex}{1.5ex}{\itshape\rmfamily}{}
{\bfseries\rmfamily}{}{2ex}{}

\newtheoremstyle{rem}{1.3ex}{1.3ex}{\rmfamily}{}
{\itshape\rmfamily}{}{1.5ex}{}
\theoremstyle{rem}
\newtheorem{remark}{{\slshape\sffamily Remark}}[]

\refstepcounter{subsubsection}

\def\thebibliography#1{\section*{References}
  \list%
  {\arabic{enumi}.}
    {\settowidth\labelwidth{[#1]}\leftmargin\labelwidth
    \advance\leftmargin\labelsep
    \parsep0pt\itemsep0pt
    \usecounter{enumi}}
    \def\newblock{\hskip .11em plus .33em minus .07em}
    \sloppy                   
    \sfcode`\.=1000\relax}

\begin{document}
\title[Gibbs measures on mutually interacting Brownian paths under singularities]
{\large Gibbs measures on mutually interacting Brownian paths under singularities}
\author[ Chiranjib Mukherjee ]{}
\maketitle
\thispagestyle{empty}
\vspace{-0.5cm}

\centerline{\sc  Chiranjib Mukherjee\footnote{Courant Institute of Mathematical Sciences, 251 Mercer Street, New York 10012, USA, {\tt mukherjee@cims.nyu.edu}}}
\renewcommand{\thefootnote}{}

\vspace{-0.5cm}
\centerline{\textit{Courant Institute New York and WIAS Berlin}}
\vspace{0.2cm}

\begin{center}
\version
\end{center}

\begin{abstract}

We are interested in the analysis of Gibbs measures
defined on two independent Brownian paths in $\R^d$ interacting through a mutual self-attraction. This is expressed by the Hamiltonian 
$\int\int_{\R^{2d}} V(x-y) \mu(\d x)\nu(\d y)$ with two probability measures $\mu$ and $\nu$ representing 
the occupation measures of two independent Brownian motions. 
We will be interested in class of potentials $V$ which are {\it{singular}}, e.g., Dirac or Coulomb 
type interactions in $\R^3$, or the correlation function of the  
parabolic Anderson problem with white noise potential.  

The mutual interaction of the Brownian paths inspires  a compactification of the 
quotient space of orbits of product measures, which is structurally different from the 
self-interacting case introduced in \cite{MV14}, owing to the lack of shift-invariant structure in the mutual 
interaction. We prove a {\it{strong large deviation principle}} for the product measures 
of two Brownian occupation measures in such a compactification, and derive 
asymptotic path behavior under Gibbs measures on Wiener paths
arising from mutually attracting singular interactions. For the spatially smoothened parabolic Anderson model with white noise potential, our analysis allows a direct computation of the annealed Lyapunov exponents
and a strict ordering of them implies the {\it{intermittency effect}} present in the smoothened model.
\end{abstract}

\maketitle   






\section{Introduction and Motivation.}\label{sectintr}

\subsection{Motivation.}

Let $\P^{\ssup 1}$ and $\P^{\ssup 2}$ be two Wiener measures corresponding to two independent Brownian motions $W^{\ssup 1}$ and $W^{\ssup 2}$ in $\R^3$. 
For a function $V: \R^3\to \R$ that vanishes at infinity, we can write down the transformed path measures
\begin{equation}\label{pathmeasure} 
\d \widehat\P_t^\otimes=  \frac 1 {Z_t} \exp\bigg\{\frac 1t \int_0^t \int_0^t V\big(\omega^{\ssup 1}_\sigma- \omega^{\ssup 2}_s\big)\,\,\d\sigma\,\d s\bigg\} \,\,\d\P^{\otimes}
\end{equation}
Here $Z_t$ is the normalizing constant that makes $\widehat\P_t^\otimes$ a probability measure and $\P^\otimes=\P^{\otimes}\otimes\P^{\ssup 2}$.
If, for each $i=1,2$,
$$
L^{\ssup i}_t=\frac 1t \int_0^t \d s \,\delta_{W^{\ssup i}_s} 
$$
denotes the normalized occupation measures of $W^{\ssup i}$ until time $t$, we can  write
$$
\d\widehat\P_t^\otimes\big(\cdot)
= \frac 1 {Z_t} \exp\big\{t H\big(L^{\otimes}_t\big)\big\} \,\,\d\P^{\otimes}(\cdot)
$$
where, for any probability measures $\mu^{\ssup 1}, \mu^{\ssup 2}\in \Mcal_1$ on $\R^3$, we denote by $\mu^{\otimes}=\mu^{\ssup 1}\otimes\mu^{\ssup 2}$ the product measure, and 
\begin{equation}\label{Hprod}
H\big(\mu^\otimes\big)=\int\int_{\R^3\times \R^3} V(x-y) \,\mu^{\ssup 1} (\d x) \mu^{\ssup 2}(\d y).
\end{equation}
Since we will be interested in functions $V(\cdot)$ which vanish at infinity, the interaction in the measures $\widehat\P_t^\otimes$ are mutually attractive. 
In particular, for certain ``singular potentials" $V$, these models  are often motivated by statistical mechanics. 
In the present article, we will study some of these models and analyze {\it{joint behavior}} of two
mutually interacting Brownian paths under Gibbs transformations of the form $\widehat\P_t^\otimes$.
Let us take a closer look at the particular models under interest.

\subsection{Dirac interaction in $\R^3$.} \label{intro-sec-Dirac}
Let $V$ be given by the Dirac measure $\delta_0$ at zero in $\R^3$. Then the path measures corresponding to 
\eqref{pathmeasure} can be written formally as
\begin{equation}\label{Qt}
\begin{aligned}
\d\mathbb Q_t&= \frac 1 {Z_t} \exp\bigg\{\frac 1t \int_0^t \int_0^t \d\sigma\, \d s\,\,\delta_0\big(\omega^{\ssup 1}_\sigma- \omega^{\ssup 2}_s\big)\bigg\} \,\d\P^{\otimes}\\
&= \frac 1 {Z_t} \exp\bigg\{\frac 1t\prod_{i=1}^2 \int_0^t \d s\,\, \delta_0(\omega^{\ssup i}_s)\bigg\}  \,\d\P^{\otimes}
\end{aligned}
\end{equation}
We remark that the interaction present in the above exponential weight is clearly {\it{self-attractive}}.  We also remark  that the factor $\frac 1t$ in the exponent in \eqref{Phat} makes the model interesting. Indeed, the double-integral in the exponent is of order $t^2$ for paths that intersect (possibly at different times), and the entropic cost for this behavior is $\e^{-O(t)}$; it is relatively easy to suspect that such a behavior is typical under the transformed measure. Hence, it is the factor $\frac1t$ that makes the energy and the entropy terms run on the same scale and still gives the path enough freedom to fluctuate.

We also that the exponential weight (ignoring the $1/t$ factor) appearing above is also intimately related to a random measure in $\R^3$, which can be written symbolically as 
\begin{equation}\label{ellt}
\ell_t(A)= \int_A\d y \,\, \prod_{i=1}^2 \int_0^t \d s \,\, \delta_y(W^{\ssup i}_s) \qquad A\subset \R^d.
\end{equation}
Then the exponential weight (modulo the $1/t$ factor) in \eqref{Qt} is the ``density" (w.r.t. the Lebesgue measure) of $\ell_t$ at zero.
Note that $\ell_t$ is a measure which is supported on the random set 
\begin{equation*}
S_t=\bigcap_{i=1}^2 W^{\ssup i}_{[0,t]}= \bigg\{x\in \R^3\colon\, x=W^{\ssup 1}(\sigma)= W^{\ssup 2}(s)\,\,\,\mbox{ for some } \sigma, s\in [0,t]\bigg\}
\end{equation*}
of intersections of the paths $W^{\ssup 1}$ and $W^{\ssup 2}$ until time $t$. It is well-known that $S_t$ is non-empty, has measure zero and Hausdorff dimension one in $\R^3$ (see \cite{DE00a}).
Note that $\ell_t$ formally counts the intensity of intersections of the two paths until time $t$ and is called the {\it{intersection local time}} or the {\it{Brownian intersection measure}} of $W^{\ssup 1}$ and $W^{\ssup 2}$.

We remark that both \eqref{Qt} and \eqref{ellt} are only formal expressions in $\R^3$. 
In one spatial dimension, these formulas can be made rigorous since the Brownian 
occupation measures of the motions have almost surely a Lebesgue density (the Brownian local time),
which is jointly continuous in the space and the time variable, and $\ell_t$ can be defined as the measure with a density
given by the pointwise product of the two occupation densities. However, in higher dimensions, occupation measures
become singular and Brownian intersection measures are difficult to work with. On a rigorous level, 
an implicit construction was carried out by Geman, Horowitz and Rosen (\cite{GHR84}) by defining $\ell_t$ as the projected local time of the {\it{confluent Brownian motion}} at zero.
Furthermore, Le Gall \cite{LG86} identified $\ell_t$ as a renormalized limit of the Lebesgue measure supported on the
intersection of the corresponding {\it{Wiener sausages}}.

In the asymptotic regime $t\to\infty$, tail behavior of the total mass $\{\ell_t(\R^3)> at^2\}$, for $a>0$, has appeared in the work of K\"onig and M\"orters (\cite{KM01}) and Chen (\cite{Ch03}),
using techniques based on fine analysis of high moments of the total mass $\ell_t(\R^3)$ 
(analyzing moments of the form $\E^\otimes\big[\ell_t(\R^3)^k\big]$ for large $k$). However, these asymptotic moment analysis 
does not allow large deviation treatment of $\ell_t$ as a {\it{measure}}.
 Such a large deviation analysis was carried out in \cite{KM13} 
for the distributions of the rescaled measures $t^{-2}\ell_t$, conditional on restricting both Brownian motions in a fixed compact set in $\R^3$.
 However, one fundamental first step in the proof of \cite{KM13} was based on classical Donsker-Varadhan theory of weak large deviations (\cite{DV75-1}-\cite{DV83-4}) leading to the stringent assumption that the motions do not leave a given bounded set until a large time and is killed upon exiting the bounded domain. Another important technical step
in \cite{KM13} was based on the exponential approximation of the intersection measure by its smoothened (mollified) version 
based on the spectral theorem and eigenvalue expansion of $-\frac 12 \Delta$ with zero boundary condition on the given bounded set.
Note that both these techniques fail when the restriction to a compact set is dropped and both Brownian motions are allowed to run freely in the full space $\R^3$.

The methods developed in this article allow a direct treatment of the measures $\ell_t$ in the {\it{full space}} $\R^3$, and proves
sharp asymptotic localization properties of the distributions of the path measures $t^{-2}\ell_t$ under the Gibbs transformations of the form \eqref{Qt}, 
see Theorem \ref{thm2-LDPBrowInt} for a precise statement. Our approach does not rely on classical Donsker-Varadhan theory, and is based on a robust theory of compactness, which builds on and extends the theory introduced in \cite{MV14}, see Section \ref{sec-heur-compact} for a heuristic discussion.

\subsection{Coulomb interaction.}\label{intro-sec-bipolaron}

Let us now consider the case of a Coulomb potential $V(x)=|x|^{-1}$ in $\R^3$ appearing in
\eqref{pathmeasure}. To motivate our work, let us 
first focus on existing work  on the study a path measure $\widehat\P_t$ defined as in \eqref{pathmeasure} with respect to a {\it{single Brownian path}} $W^{\ssup 1}= W^{\ssup 2}=W$, i.e., 
 \begin{equation}\label{Phat}
 \begin{aligned}
\widehat{\P}_t(\d \omega)&=\frac 1 { Z_t}\, \exp\bigg\{\frac 1t\int_0^t\int_0^t \d \sigma\d s \,\frac 1{\big|\omega_\sigma-\omega_s\big|}\bigg\} \, \P(\d \omega) \\
&= \frac 1 { Z_t}\, \exp\big\{t H(L_t)\big\} \, \d \P,
\end{aligned}
\end{equation} 
where $L_t=\frac 1t\int_0^t\d s\, \delta_{W_s}$ stands for the normalized occupation measure, for which $H(L_t)=\int\int_{\R^3\times\R^3} |x-y|^{-1} L_t(\d x) L_t(\d y)$ and $Z_t=\E[\exp\{t H(L_t)\}]$, and $\P$ refers to 
the three dimensional Wiener measure for $W$. 

The infinite volume limit of $\widehat\P_t$ as well as the distributions $\widehat\P_t\circ L_t^{-1}$ have been completely analyzed in a series of recent work (see \cite{MV14}, \cite{KM15} and \cite{BKM15}) .
The main impetus for this work came from the {\it{polaron problem}} in quantum statistical mechanics, where one is interested in the behavior of a charged particle 
(an electron) whose movement is described by the measures
\begin{equation}\label{polaron}
\widehat\P_{\lambda,t}(\d \omega)= \frac 1 {Z_{\lambda,t}} \exp\bigg\{\lambda \int_0^t\int_0^t \d\sigma \d s \frac{\e^{-\lambda|\sigma-s|}}{|\omega_\sigma- \omega_s|}\bigg\} \,\, \P(\d \omega),
\end{equation}
where $\lambda=\alpha^{-2}>0$ is a parameter and $Z_{\lambda,t}$ is the normalization constant or partition function. 
The physically relevant regime is the  {\it{strong coupling limit}} as $\alpha\to \infty$, i.e., $\lambda\to0$.
The asymptotic behavior of the partition function $Z_{\lambda,t}$ in the limit $t\to\infty$, followed by $\lambda\to0$, was rigorously studied  by Donsker and Varadhan (\cite{DV83-P}),
 proving the following variational formula for the free energy: 
\begin{equation}\label{Pekarfor}
 \begin{aligned}
\lim_{\lambda\to 0}\lim_{t\to\infty}\frac 1t\log Z_{\lambda,t}&= \lim_{t\to\infty}\frac 1t\log Z_t \\
&=\sup_{\heap{\psi\in H^1(\R^3)}{\|\psi\|_2=1}} 
\Bigg\{\int_{\R^3}\int_{\R^3}\d x\d y\,\frac {\psi^2(x) \psi^2(y)}{|x-y|} -\frac 12\big\|\nabla \psi\big\|_2^2\Bigg\},
\end{aligned}
\end{equation}
where $Z_t$ is the normalizing constant in \eqref{Phat}
and $H^1(\R^3)$ denotes the usual Sobolev space of square integrable
functions with square integrable gradient. The variational formula in \eqref{Pekarfor} was analyzed by Lieb (\cite{L76}) with the result that there is a rotationally symmetric maximizer $\psi_0$, which is unique modulo spatial shifts. 
Given the coincidental behavior \eqref{Pekarfor} of the partition functions, one believes that the polaron measure $\widehat\P_{\lambda,t}$ to be 
approximated by $\widehat\P_t$ for $\lambda\to 0$. {\footnote{Note that the interaction appearing in \eqref{polaron} is self-attractive, the measure favors paths which, at least for fixed $\lambda>0$ and short time scales,
tend  to clump together. However, for the $\lambda\sim 0$ (i.e., strong coupling regime), this attractive effect gets more and more smeared out 
and the behavior of the actual path measures $\widehat\P_{\lambda,t}$ is then believed to be an approximation of the limiting asymptotic behavior of $\widehat\P_t \circ L_t^{-1}$, see the 
discussion in Section 1.1 in \cite{BKM15}, as well as (Spohn \cite{Sp87}).}} This quest led to a rigorous analysis 
of $\widehat\P_t$ and the first fundamental result (\cite{MV14}) showed that the asymptotic distribution of $L_t$ under $\widehat\P_t$
concentrates around the maximizers $\mathfrak m=\{\psi_0^2\star\delta_x\colon x \in \R^3\}$ of the variational problem \eqref{Pekarfor}, i.e., 
\begin{equation}\label{tubeweak}
\limsup_{t\to\infty}\frac 1t \log \widehat\P_t\big\{L_t\notin U(\mathfrak m)\big\}<0.
\end{equation}
This result was further sharpened in \cite{BKM15} showing that identifying the exact limiting distributions of $\widehat\P_t\circ L_t^{-1}$ as an explicit mixture of the maximizers in $\mathfrak m$,
thus, contributing on a rigorous level to the aforementioned 
heuristic understanding of the ``mean-field approximation of the polaron problem" on the level of path measures.

Let us now turn to the measures $\widehat\P_t^\otimes$  pertaining two independent Brownian motions $W^{\ssup 1}$ and $W^{\ssup 2}$, 
defined in \eqref{pathmeasure} for $V(x)=1/|x|$. This is related to the study of {\it{bi-polarons}} in statistical mechanics, see the work of Miyao-Spohn (\cite{MS07}) and Frank-Lieb-Seiringer (\cite{FLS13}).
The model is written by the path measures
\begin{equation}\label{bipolaron}
\begin{aligned}
\d\,\,\widehat\P_{\lambda, u, t}^\otimes&= \frac 1 {Z_{\lambda,u,t}} \exp\bigg\{\sum_{i,j=1}^2 \bigg(\frac\lambda 2 \int_0^t\int_0^t \d\sigma \d s \frac{\e^{-\lambda|\sigma-s|}}{|\omega^{\ssup i}_\sigma- \omega^{\ssup j}_s|}\bigg) \\
&\qquad\qquad\qquad\qquad\qquad\qquad- u \bigg(\sum_{i<j} \int_0^t \d s \frac 1{\big|\omega^{\ssup i}- \omega^{\ssup j}_s\big|}\bigg)\bigg\} \,\,\d\P^\otimes 
\end{aligned}
\end{equation}
with $\lambda,u>0$. These measures describe the motion of two electrons coupled to an ionic crystal, and unlike the case of a single polaron, 
the interaction in \eqref{bipolaron} carries both attractive (contributions coming from the first term) and repulsive (the Coulomb 
force appearing in the second term) effects. Again for strong coupling regime, the behavior of the partition function $\lim_{\lambda\to 0}\lim_{t\to\infty} Z_{\lambda,u,t}= E(u)$
is expressed by the {\it{Pekar-Thomasevich}} energy functional (\cite{MS07}, \cite{FLS13}), and $E(u)$ determines, depending on the value of $u$, if in 
the lowest energy states two electrons form a bound pair, or they split into two well-separated polarons. Again questions concerning the behavior 
of the actual path measures $\widehat\P_{\lambda, u, t}^\otimes$ remain widely open. However, in \eqref{bipolaron}, intuitively one believes that
for small $\lambda$ and fixed $u$, the attractive interaction in the first term in the exponential should get more and more smeared out and again should approach the mean-field model
$$
\frac 1{2t}\sum_{i,j}\int_0^t\int_0^t \d \sigma\d s \,\frac 1{\big|\omega^{\ssup i}_\sigma-\omega^{\ssup j}_s\big|}, 
$$
that corresponds to the interaction appearing in $\widehat\P_t^\otimes$, recall \eqref{pathmeasure}. Motivated by this intuition,  
our first result shows that the asymptotic distribution of the product occupation measures $L_t^\otimes$ under $\widehat\P_t^\otimes$ is attracted by the ``infinite tube" of spatial shifts of a centered maximizer of a variational formula similar to \eqref{Pekarfor},
being in analogy to the fundamental result \eqref{tubeweak}. The precise statement 
of this result can be found in Theorem \ref{thmpairtube}. Given the understanding of a single polaron $\widehat\P_{\lambda,t}$
via its mean-field approximation $\widehat\P_t$ as we discussed before, it is conceivable that this result will lead to a clearer understanding of the bipolaron problem on the level of
path measures.

\subsection{Parabolic Anderson model with white noise potential.}\label{intro-sec-PAM}
The third motivation of our work comes from studying the moments of the approximating solutions to the ill-defined stochastic partial differential equation
\begin{equation}\label{pamintro}
\partial_t Z= \frac 12 \Delta Z+ \eta Z
\end{equation}
for spatial white noise potential $\eta$ in $\R^3$, i.e., $\eta$ is Gaussian noise in $\R^3$ with covariance kernel 
$\E(\eta(x) \, \eta(y))= \delta_0(x-y)$. This is called the spatial {\it{parabolic Anderson problem}}. 
A rigorous construction of this ill-posed equation has been carried out in \cite{HL15} based on the robust 
theory of regularity structures (\cite{H14}). 

On the other hand, when $\eta$ is given by a more regular random potential, the parabolic Anderson model \eqref{pamintro}
 is well-defined and has been extensively studied recently; see the monograph (K\"onig \cite{K16}) for a comprehensive summary of these results, concepts of the proofs, 
and much more related material. In particular, for the spatially continuous situation, when $\eta$ is a regular shift-invariant 
potential (Gaussian or Poisson shot noise) the second order asymptotics of the $p$-th moment ($p\geq 1$)
of $Z(t,0)$ as $t\to\infty$, as well as the almost sure asymptotics of the second order term of $Z(t,0)$ have been determined (\cite{GK00}, \cite{GKM00}),
see also the related work of Cranston-Mountford-Siga (\cite{CMS02}) and Cranston-Mountford (\cite{CM06}).
A prominent feature exhibited by this model is a strong localization effect, called {\it{intermittency}}: the random path has an inherent propensity to be confined
to some few small islands in the random medium, which are widely spread, and the main mass of the solution is built up in these islands. Therefore, the global 
property of this system is determined by the local extremes of the random field, rather than an averaging behavior coming from a central limit theorem. 
On a technical level, the proofs of these results are crucially based on a Feynman-Kac  representation of the solutions in terms of an exponential integral w.r.t. the Brownian path,
and applying classical (weak) Donsker-Varadhan large deviations for the distributions of local times of the path restricted to a large bounded box.

When $\eta$ represents a singular white noise potential in space, the method presented in this article allows a direct computation of the annealed (i.e., averaged over the noise) moments of a smoothened and rescaled version
of \eqref{pamintro} , as the smoothing parameter goes to zero. These Lyapunov exponents admit explicit variational formulas and a strict ordering
of these exponents underlines the aforementioned intermittent behavior of the smoothened solutions, see Theorem \ref{Lyapexp} and Lemma \ref{pam-lemma}.
Our proof techniques  are not reliant on classical Donsker-Varadhan theory and are based on a robust theory of ``compact large deviations" developed in Section \ref{sec-paircompact}. 
It is relatively easy to see that in case of a more regular shift-invariant random potential $\eta$ 
(and thus, a well-defined model \eqref{pamintro} that requires no mollification procedure),
the same arguments appearing in the proof of Theorem \ref{Lyapexp} lead to the asymptotic analysis of the path behavior of the model \eqref{pamintro} as $t\to\infty$ 
that recover the main assertions of \cite{GK00} and \cite{GKM00} (without restricting paths to any bounded domain).

\subsection{An essential ingredient: Compactness and strong large deviations:} \label{sec-heur-compact}
For studying asymptotic behavior of path behavior of Gibbs measures defined on interacting Wiener paths, it is desirable to have
a {\it{full large deviation principle}} for the distributions of occupation measures $L_t$ of a 
Brownian motion in the space $\Mcal_1(\R^3)$ of probability measures. Note that this space is non-compact in the usual weak topology and 
in absence of positive recurrence, classical theory of large deviations is restricted only to compact subsets of $\Mcal_1(\R^3)$.
Such a problem resisted a full resolution for studying statistical mechanical models on the level of path measures for a long time until a robust theory of  `` compactification"  was
introduced in \cite{MV14} for models that are inherently shift-invariant. For instance, note that for the measures $\widehat\P_t$ in \eqref{Phat}, 
\begin{equation}\label{shiftinv}
H(L_t)= \int\int V(x-y) \, L_t(\d x) \,L_t(\d y)=H(L_t\star \delta_z) \qquad\forall \, z\in \R^3.
\end{equation}
For models that enjoy such shift-invariance, it is natural to compactify the quotient space $\widetilde\Mcal_1(\R^3)$ of orbits $\widetilde\mu=\{\mu\star\delta_x\colon \, x\in \R^3\}$ 
probability measures under translations, and a full large deviation principle for the distributions of the orbits $\widetilde L_t$ embedded in that compactification
allows one to prove assertions on path behavior of Gibbs measures, thanks to the shift-invariance property \eqref{shiftinv}.

In the present context, for our main assertionsconcerning the models introduced in Section \ref{intro-sec-Dirac}-\ref{intro-sec-PAM}, 
it is tempting to appeal to the methods developed in \cite{MV14}. However, note that due to the ``mixed product" 
of two different measures in \eqref{Hprod}, the crucial translation-invariance of \eqref{shiftinv} is somewhat lost:  $H(\mu\otimes \nu)= \int\int V(x-y)\mu(\d x)\nu (\d y)$ might fail to be equal to 
$$
H\big(\,(\mu\star \delta_{z_1})\otimes (\nu\star\delta_{z_2}) \,\big)
$$ 
if $z_1\ne z_2$. Hence, such a mixed product disallows a straightforward product measure compactification of $\widetilde\Mcal_1(\R^3) \, \otimes \widetilde\Mcal_1(\R^3)$
and we are not entitled to invoke the shift-invariant theory developed in \cite{MV14}.


However, the central idea of \cite{MV14} inspires the following resolution. Note that any sequence of probability measures 
in $\R^d$ might fail to have a convergent subsequence in the usual weak topology, since masses may escape in different
directions (e.g., $\mu_n= \frac 12 \delta_n + \frac 12 \delta_{-n}$ with $n\to\infty$), or a sequence of measures could totally disintegrate into dust (like a centered Gaussian with a large variance).
Starting with any arbitrary sequence $\mu_n $ of probability measures,
the first step is to identify local regions in $\R^d$ where $\mu_n$ and $\nu_n$ have their accumulations of masses. This can be written by the concentration function
$q_{\mu_n}(r)= \sup_{x\in \R^d} \mu_n\big(B_r(x)\big)$ of $\mu_n$.
By choosing subsequences, we can assume that the limits $q_{\mu_n}(r)\to q_\mu(r)$ as $n\to\infty$ and $q_\mu(r)\to p_1 \in [0,1]$ as $r\to\infty$
exist. Restricting $\mu_n$ to such a suitable concentration region, we can take a shift $\mu_n\star \delta_{a_n}$ which converges vaguely along some subsequence to a sub-probability measure $\alpha_1$ of mass $p_1$. We can peel off a measure of mass $\approx p_1$ from $\mu_n$ and repeat the above process now for the leftover to get convergence again along a further subsequence. We can go on recursively and end up with a picture where $\mu_n$ concentrates roughly on 
{\it{widely separated}} compact pieces of masses $\{p_j\}_{j\in\N}$ with $p_j\downarrow 0$ and $\sum_j p_j \leq 1$, while the rest of the mass $1-\sum_j p_j$ leaks out. 

In an exactly similar manner, for an independent sequence of measures $\nu_n$, 
we can also visualize a similar concentration in widely separated compact pieces of masses $\{q_l\}_{l\in \N}$ 
(with a remaining mass $1-\sum_l q_{l}$ being dissipated). 
However, concentration regions of $\mu_n$, and that of and $\nu_n$, are a-priori completely independent and hence, are also possibly mutually widely separated. 
This is a problem that one encounters to recover any partial mass (in the limit) for the product sequence $\mu_n\otimes\nu_n$.

In our desired compactification, we will be interested in a topology where functionals $\mu\otimes\nu\mapsto \int V(x-y) \mu(\d x) \nu(\d y)$ are continuous,
with test functions $V$ vanishing at infinity. Hence, it is conceivable that, one can hope to recover any partial mass from the product sequence $\mu_n\otimes\nu_n$, only if some concentration region 
of $\mu_n$ happens to have be in a bounded distance from some concentration region of $\nu_n$. Since there is wide separation between each individual components of $\mu_n$ and 
individual components of $\nu_n$, such ``matchings" could take place only pairwise. Since the test functions $V$ vanish at infinity, 
any other mutually distant (and ``unmatched") pair of components can not contribute. We end with a picture where the product sequence $\mu_n\otimes\nu_n$ roughly concentrates on 
matched pairs of islands, so that within each pair, two components are within bounded distance, while the pairs are mutually widely separated, with a certain dissipation of mass
coming from unmatched pairs. The space of all such matched pairs of components will be in our desired  ``compactification". {\footnote{For example, let $\mu_n$ be a sequence which is a mixture of three Gaussians, one with mean $0$ and variance $1$, one with mean $n$ and variance $1$ and one with mean $0$ and variance $n$, each with equal weights $1/3$. On the other hand, let $\nu_n$ also be a mixture of three Gaussians, one with mean $n^2$ and variance $1$, one with mean $n+1$ and variance $1$ and one with mean $0$ and variance $n$, also with equal weights $1/3$. Then the limiting object for $\widetilde{\mu_n\nu_n}$ is the single orbit $\{\widetilde{\alpha_1\beta_2}\}$, where $\alpha_1$ is the a Gaussian with mean $0$, variance $1$ and mass $1/3$, $\beta_2$ is a Gaussian with mean $1$, variance $1$ and mass $1/3$, while $\widetilde{\alpha_1\beta_2}$ is the equivalence class of the product of these two Gaussians with mass $1/9$.}

In such a compactification, due to mutual attraction, two independent Brownian paths tend to find such matched pairs of islands and stick together by treating
each pair as one bigger island. Since such bigger islands are mutually distant, an optimal strategy rules out any interaction 
between them, leading to asymptotic independence and a full large deviation principle for distribution of orbits 
$\{L_t^{\otimes} \star \delta_x\colon x \in \R^d\}$
embedded in the compactfication, the rate function simply being
the sum of classical Donsker-Varadhan rate functions on each such island. This  heuristic idea will be a key recipe in our study.

Let us summarize the organization for the rest of the article. In Section \ref{sec-paircompact}, we describe a precise layout for the compactifcation procedure described above.
Section \ref{sec-BrowInt} is devoted 
to the study of path measures with respect to Dirac type interactions in $\R^3$ discussed in Section \ref{intro-sec-Dirac}. 
It is worth pointing out that, due to the singular nature of the Dirac interactions in three dimensions, 
the machinery developed in Section \ref{sec-paircompact} is a priori not applicable to the measures appearing for this model.
and one needs substantial work to show that, after a suitable mollification procedure, such singularities can be tamed down
on an exponential scale, see Section \ref{sec-BrowInt-proof}.
In Section \ref{sec-pairpolaron} we prove localization of path measures for the Coulomb interaction introduced in Section \ref{intro-sec-bipolaron},
and in  \ref{sec-PAM}, we compute the annealed Lyapunov exponents of the parabolic Anderson problem introduced in Section \ref{intro-sec-PAM}.

\section{Shift invariant compactification for product measures}\label{sec-paircompact}

In this section we will work with any arbitrary spatial dimension $d$ and write $\Mcal_1= {\Mcal_1}(\R^d)$ the space of probability distributions on $\R^d$,
tacitly equipped with the weak topology which dictates a 
sequence $\mu_n$ in $\Mcal_1$ to converge to $\mu$,  denoted by  $\mu_n\weak\mu$,  if
\begin{equation}\label{weakconv}
\lim_{n\to\infty}\int_{\R^d} f(x)\mu_n(\d x)=\int_{\R^d} f(x) \mu(\d x),
\end{equation}
for all bounded continuous functions on $\R^d$.  If we denote by $\Mcal_{\leq 1}$ the space of all sub-probability measures (non-negative measures with total mass less than or equal to one), then 
the same topology carries over to this space with the same requirements. In Section \ref{sec-2.1} we will introduce a ``diagonal shift-invariance" property of the product space $\Mcal_1^\otimes=\Mcal_1\otimes\Mcal_1$ 
of probability measures and describe certain continuous functionals on the corresponding quotient space $\widetilde\Mcal_1^\otimes$ under that shift-invariant action. In Section \ref{sec-2.2} we will introduce a metric 
space $\widetilde{\mathcal X}^\otimes\hookleftarrow\widetilde\Mcal_1^\otimes$ which is indeed the compactification and completion of the totally bounded quotient space $\widetilde\Mcal_1^\otimes$, see Theorem \ref{thmcompact}. In Section \ref{sec-compactldp} we will prove a strong large deviation principle for the orbits of product occupation measures $L_t^\otimes$ embedded in that compactification.


\subsection{Diagonal shift-invariance on product measures.}\label{sec-2.1}
We turn to the product space 
$$
\Mcal_1^\otimes=\Mcal_1\otimes \Mcal_1
$$ 
and its quotient space
$$
\begin{aligned}
\widetilde\Mcal_1^\otimes&={(\Mcal_1\otimes\Mcal_1)}\big/ \sim\\
&= \big\{(\mu\star\delta_x)\otimes(\nu\star\delta_x)\colon x \in \R^d\big\}
\end{aligned}
$$
of {\it{diagonal orbits}} under translations by $\R^d$. 
For notational convenience, we will write $\mu\nu= \mu\otimes \nu$ for typical elements of $\Mcal_1^\otimes$ and 
$\widetilde{\mu\nu}$ for its  diagonal orbit in $\widetilde\Mcal_1^\otimes$. 
The goal of this section is to provide with a compactification for the space $\widetilde\Mcal_1^\otimes$, which is a priori 
non-compact under the product and quotient operations and the weak topology inherited from $\Mcal_1$.

\noindent First we need to identify a class of continuous functionals on this space.
 For any $k\geq 1$, we denote by $\mathcal F_k^{\otimes}$ the space of continuous functions $f: (\R^{2d})^k \rightarrow \R$ that vanish at infinity in the sense 
\begin{equation}\label{vanishinfty}
\lim_{r\to\infty} f\big(x_1,y_1,\dots,x_k,y_k\big)=0.
\end{equation}
where 
$r$ is the diameter given by $\max_{i,j=1,\dots,k}\big\{|x_i-y_j|, |x_i-x_j|, |y_i-y_j|\big\}$.
Furthermore, these functions should be {\it{diagonally
translation-invariant}} in the sense 
$$
\begin{aligned}
f\big(x_1,y_1,\dots,x_k,y_k\big)&= f\big(x_1+a,y_1+a,\dots,x_k+a,y_k+a\big)\\
&\qquad\qquad\qquad\forall\,\,(x_j, y_j)\in \R^{2d},\,\, a\in \R^d, \, j=1,\dots, k
\end{aligned}
$$

 We will often use a typical function $f\in\mathcal F_1^{\otimes}$ of the form
 $f(x_1,y_1)=V(x_1-y_1)$ for some continuous function $V(\cdot)$ vanishing at infinity. 
 Denote by $\mathcal F^{\otimes}=\cup_{k\geq 1} \mathcal F_k^{\otimes}$
 the countable union. 
 
 \subsection{Compactification with respect to diagonal shift-invariance.}\label{sec-2.2}
  We define the space
  \begin{equation}\label{Xotimes}
 \begin{aligned}
 {\widetilde{\mathcal X}^\otimes}= \bigg\{\xi^{\otimes}\colon \, \xi^{\otimes}=\big\{\widetilde{\alpha_i\beta_i}\big\}_{i\in I}, \,\,\,\alpha_i,\beta_i\in \Mcal_{\leq 1}.
 \sum_i \alpha_i(\R^d)\leq 1, \sum_i \beta_i(\R^d)\leq 1 \bigg\}
 \end{aligned}
 \end{equation}
 of all empty, finite and countable collection of diagonal orbits  $\xi^\otimes=\{\widetilde{\alpha_j\beta_j}\}_j$ 
 of products of measures with their masses $\sum_j \alpha_j(\R^d)$ and $\sum_j \beta_j(\R^d)$ 
 adding up to most one. Note that, we can canonically identify elements of $\widetilde{\Mcal}_1^{\otimes}$ in the space $\widetilde{\mathcal X}^\otimes$. 
 
We would like a metric on the space  $\widetilde{\mathcal X}^\otimes$. 
We set, for any $\xi^\otimes\in \widetilde{\mathcal X}^\otimes$, any $f\in \mathcal F^\otimes$,
$$
\Lambda^\otimes(f,\xi^\otimes)= 
\sum_{\widetilde{\alpha\beta}\in\xi^\otimes}\int f(x_1,y_1\ldots,x_k,y_k)\prod_{i=1}^n\big(\alpha\beta\big)(\d x_i \d y_i)
$$
Due to shift-invariance of $f\in\mathcal F^\otimes$, each summand in the above expression depends only on the orbit $\widetilde{\alpha\beta}= \{(\alpha\star\delta_x)\otimes(\beta\star \delta_x) \colon x \in \R^d\}$. 
We fix a countable sequence  of functions $\{f_r(x_1y_1,\ldots, x_{k_r}, y_{k_r})\}_{r\in \N}$ which is dense in $\mathcal F^\otimes$. For any $\xi_1^\otimes,\xi_2^\otimes\in {{\widetilde{\mathcal X}^\otimes}}$, we then 
define
\begin{equation}\label{Ddef}
\mathbf D^\otimes\big(\xi_1^\otimes,\xi_2^\otimes\big)=\sum_{r=1}^\infty\frac{1}{2^r}\frac{1}{1+\|f_r\|_\infty}\,\, \big| \Lambda^\otimes(f_r,\xi_1^\otimes)- \Lambda^\otimes(f_r,\xi_2^\otimes)\big|.
\end{equation}
To see that $\mathbf D^\otimes$ is a metric on ${\widetilde{\mathcal X}^\otimes}$, note that $\mathbf D^\otimes$ clearly satisfies the triangle inequality as well as positivity. 
The proof of the fact that $\mathbf D^\otimes\big(\xi_1^\otimes,\xi_2^\otimes\big)=0$ implies $\xi_1^\otimes=\xi_2^\otimes$ is non-trivial and lengthy, but follows the same line of arguments modulo slight modifications as Theorem 3.1 in \cite{MV14}.  Note that for any $f\in \mathcal F^\otimes$, $\Lambda^\otimes(f, \cdot)$ is a natural continuous functional on the space $\widetilde{\mathcal X}^\otimes$, in the metric
$\mathbf D^\otimes$, the sequence $(\xi_n^\otimes)_n$ converges to $\xi^\otimes$ in the space ${\widetilde{\mathcal X}^\otimes}$, if the sequence
\begin{equation}\label{Dconverge}
\begin{aligned}\Lambda^\otimes(f,\xi_n^\otimes)
&=\displaystyle\sum_{(\widetilde{\alpha_n\beta_n})\in\xi_n^\otimes}\int f(x_1,y_1\ldots,x_k,y_k)\prod_{i=1}^n\big(\alpha_n\beta_n\big)(\d x_i \d y_i)\\
&\to \sum_{\widetilde{\alpha\beta}\in\xi^\otimes}\int f(x_1,y_1\ldots,x_k,y_k)\prod_{i=1}^n\big(\alpha\beta\big)(\d x_i \d y_i)=\Lambda^\otimes(f,\xi^\otimes)
\end{aligned}
\end{equation}
for every $f\in\mathcal F^\otimes$.

The following statement is our main compactification result.

\begin{theorem}\label{thmcompact}
Given any sequence $(\widetilde{\mu_n\nu_n})_n$ in $\widetilde{\mathcal M}_1^\otimes$, there is a subsequence that converges to a limit in ${\widetilde{\mathcal X}^\otimes}$.
Furthermore, the set of orbits $\widetilde{\mathcal M}_1^\otimes$ is dense in ${\widetilde{\mathcal X}^\otimes}$. Hence ${\widetilde{\mathcal X}^\otimes}$ is a compactification of $\widetilde{\mathcal M}_1^\otimes$. It is then also the completion under the metric $\mathbf D^\otimes$ of the totally bounded space $\widetilde{\mathcal M}^\otimes$.
\end{theorem}

\begin{proof} We  first show that any sequence $\widetilde{\mu_n\nu_n}$ in $\widetilde\Mcal_1^\otimes$ finds a subsequence which converges to some $\xi^\otimes$
in the metric $\mathbf D^\otimes$. Let us start with the {\it{concentration functions}}
of $\mu_n$ and $\nu_n$ given by
$$
Q_{\mu_n}(r)= \sup_{x\in \R^d} \mu_n (B_r(x)) \qquad Q_{\nu_n}(r)= \sup_{x\in \R^d} \nu_n (B_r(x)).
$$
We can assume that along some subsequences, $Q_{\mu_n}(r)\to Q_\mu(r)$ and $Q_{\nu_n}(r)\to Q_\nu(r)$ as $n\to \infty$. Furthermore, 
$Q_\mu(r)\to p_1$ and $Q_\nu(r)\to q_1$ as $r\to \infty$. 

Let us first assume that $p_1, q_1\in (0,1]$. Then there is a sequence of shifts $(a_n)\subset \R^d$ such that, for some $r>0$ and $n$ sufficiently large,
\begin{equation}\label{eq1thm}
\mu_n\big(B_r(a_n)\big) \geq p_1/2.
\end{equation}
Then we can decompse 
$$
\mu_n= \alpha_n + \mu_n^{\ssup 1}
$$ 
so that by Prohorov's theorem, along some subsequence, $\alpha_n\star \delta_{a_n}\weak \alpha^\prime$ for some 
$\alpha^\prime\in \Mcal_{\leq 1}$ and $a_n\in \R^d$. Recall that \lq$\weak$\rq denotes the usual weak convergence of (sub)-probability measures.
One should also think that the sequences $\alpha_n$ and $\mu_n^{\ssup 1}$ are ``widely separated" in the sense, 
$$
\int_{\R^{2d}} V(x-y)  \alpha_n(\d x) \mu_n^{\ssup 1}(\d y)= 0,
$$
for any continuous function $V$ vanishing at infinity (i.e., $V\in \mathcal F_1^\otimes$). 

We repeat the procedure with $\mu_n^{\ssup 1}\star \delta_{a_n}$. Since $\big(\mu_n^{\ssup 1}\star \delta_{a_n}\big)(\cdot) \leq \mu_n(\cdot)$, we conclude, by \eqref{eq1thm}, that
for every $r>0$, 
$$
\lim_{n\to\infty} Q_{(\mu_n^{\ssup 1}\star \delta_{a_n})} (r) \leq \min \big\{1-p_1/2, p_1\}.
$$
This iterative process could go on forever, or it might stop at a finite stage (i.e., when the recovered mass $p_{k+1}$, after stage $k$, happens to be $0$). If it stops at a finite stage, then we can write 
$$
\mu_n= \sum_{j=1}^k \alpha_n^{\ssup j}+ \gamma_n
$$
such that, for each $j=1,\dots,k$, along some subsequence, 
\begin{equation}\label{eq2thm}
\alpha_n^{\ssup j} \star \delta_{a_n^{\ssup j}} \weak \alpha_j^\prime,
\end{equation} 
and for $i\ne j$, $|a_n^{\ssup i}- a_n^{\ssup j}| \to \infty$, while the sequence
$\gamma_n$ totally disintegrates, i.e., for every $r>0$, 
$$
Q_{\gamma_n}(r)= \sup_{x\in \R^d} \gamma_n(B_r(x)) \to 0.
$$
It is easy to deduce that the above property implies that for such a totally disintegrating sequence $(\gamma_n)_n$ and 
any arbitrary sequence $(\eta_n)_n$ of measures in $\Mcal_{\leq 1}$, and for every $k\geq 1$ and $f\in \mathcal F^\otimes_k$,
\begin{equation}\label{disint}
\lim_{n\to\infty}\int_{\R^{2dk}} f(x_1,y_1\ldots,x_k,y_k) \prod_{i=1}^k\gamma_n(\d x_i)\prod_{i=1}^k \eta_n(\d y_i)=0.
\end{equation}
Let us now turn to the sequence $(\nu_n)$ in the product $(\mu_n\otimes \nu_n)$. In an exactly similar manner, we can write,
$$
\nu_n= \sum_{l=1}^m \beta_n^{\ssup l}+ \lambda_n
$$
such that, for each $l=1,\dots,m$, along some subsequence, 
\begin{equation}\label{eq3thm}
\beta_n^{\ssup l} \star \delta_{b_n^{\ssup l}} \weak \beta_l^\prime,
\end{equation}
and for $l\ne u$, $|b_n^{\ssup l}- b_n^{\ssup u}| \to \infty$, while the sequence
$\lambda_n$ totally disintegrates and satisfies a similar requirement \eqref{disint}. 
Let us now turn to the product
\begin{equation}\label{proddecomp}
\begin{aligned}
\mu_n\otimes \nu_n&=\bigg(\sum_{j=1}^k \alpha_n^{\ssup j}+ \gamma_n\bigg) \bigg(\sum_{l=1}^m \beta_n^{\ssup l}+ \lambda_n\bigg)
\\
 &=\sum_{j=1}^k\sum_{l=1}^m \alpha_n^{\ssup j}\otimes \beta_n^{\ssup l}+ \sum_{l=1}^m \beta_n^{\ssup l}\otimes\gamma_n+ \sum_{j=1}^k \alpha_n ^{\ssup j}\otimes \lambda_n
\end{aligned}
\end{equation}
Then \eqref{disint} implies that the last two summands on the right hand side will vanish when integrated with respect to any $f\in \mathcal F_k^\otimes$, i.e., 
\begin{equation}\label{eq4thm}
\begin{aligned}
&\lim_{n\to\infty}\int_{(\R^{2d})^k} f(x_1,y_1,\dots, x_k,y_k)\prod_{i=1}^k\alpha_n^{\ssup j}(\d x_i) \prod_{i=1}^k\lambda_n(\d y_i)=0 \qquad\forall j=1,\dots,k \\
&\lim_{n\to\infty}\int_{(\R^{2d})^k} f(x_1,y_1,\dots, x_k,y_k)\prod_{i=1}^k\beta_n^{\ssup j}(\d x_i) \prod_{i=1}^k\gamma_n(\d y_i)=0 \qquad\forall l=1,\dots,m.
\end{aligned}
\end{equation}
Hence we focus only on the first summand on the right hand side of \eqref{proddecomp}. For these products $\alpha_n^{\ssup j}\otimes \beta_n^{\ssup l}$, if for some $j\in \{1,\dots,k\}$ and $l\in \{1,\dots, m\}$, the distance of the shifts $|a_n^{\ssup j}- a_n^{\ssup l}|$ remains bounded, i.e, 
\begin{equation}\label{eq5thm}
|a_n^{\ssup j}- a_n^{\ssup l}| \leq c^{\ssup{jl}}\leq C <\infty,
\end{equation}
then we can find some common spatial shift $c^{\ssup{jl}}_n=c_n$ so that, again along sone subsequence, 
$$
\begin{aligned}
\big(\alpha_n^{\ssup j}\otimes \beta_n^{\ssup l}\big)\star \delta_{c_n}&=\big(\alpha_n^{\ssup j}\star c_n\big)\otimes \big(\beta_n^{\ssup l}\star c_n\big) \\
&\weak \alpha_j \otimes \beta_l,
\end{aligned}
$$
for some $\alpha_j, \beta_l \in \Mcal_{\leq 1}$. In other words, for any such pair $j\in \{1,\dots,k\}$ and $l\in \{1,\dots, m\}$, for every $V\in \mathcal F_1^\otimes$,
$$
\begin{aligned}
\int\int V(x-y) \alpha_n^{\ssup j}(\d x)\beta_n^{\ssup l}(\d y&)= \int\int V(x-y) \big(\alpha_n^{\ssup j}\star \delta_{c_n}\big)(\d x)\big(\beta_n^{\ssup l}\star \delta_{c_n}\big)(\d y) \\
&\to \int\int V(x-y)\alpha_j(\d x)\beta_l (\d y),
\end{aligned}
$$
and likewise, for any $f\in \mathcal F_k^\otimes$,
\begin{equation}\label{eq6thm}
\begin{aligned}
&\int_{\R^{2d k}} f(x_1,y_1,\dots,x_k,y_k)\prod_{i=1}^k\alpha_n^{\ssup j}(\d x_i)\prod_{i=1}^k\beta_n^{\ssup l}(\d y_i)
\\
&\to \int_{\R^{2dk}} f(x_1,y_1,\dots,x_k,y_k)\prod_{i=1}^k\alpha_j(\d x_i)\prod_{i=1}^k\beta_l (\d y_i).
\end{aligned}
\end{equation}
Let us finally turn to the case when $|a_n^{\ssup j}- b_n^{\ssup l}|\to \infty$, as $n\to\infty$ for any $j$ and any $l$. Then, again for any $V\in \mathcal F_1^\otimes$,
\begin{equation}\label{eq7thm}
\begin{aligned}
\int\int V(x-y) \alpha_n^{\ssup j}(\d x)\beta_n^{\ssup l}(\d y)&= \int\int V\big(x-y+a_n^{\ssup j}- b_n^{\ssup l}\big) \big(\alpha_n^{\ssup j}\star \delta_{a_n^{\ssup j}}\big)(\d x)\big(\beta_n^{\ssup l}\star \delta_{b_n^{\ssup l}}\big)(\d y) \\
&\to 0
\end{aligned}
\end{equation}
since $V$ vanishes at infinity. 

Let us now summaraize \eqref{eq4thm}, \eqref{eq6thm} and \eqref{eq7thm}, and recall the definition of the metric $\mathbf D^\otimes$ from \eqref{Ddef} and the requirement \eqref{Dconverge}. We conclude that
$$
\widetilde{\mu_n\nu_n} \longrightarrow \big\{\big(\widetilde{\alpha_j\beta_l}\big)_{j,l}\big\}\in \widetilde{\mathcal X}^\otimes.
$$

Let us now consider the case $p_1=0$ or $q_1=0$. If $p_1=0$, then, for every $r>0$,
$$
\lim_{n\to \infty} \sup_{x\in \R^d} \mu_n(B_r(x))= 0.
$$
and the sequence $\mu_n$ totally disintegrates, and by \eqref{disint}, the sequence $\widetilde{\mu_n\nu_n} \to 0$ in $\widetilde{\mathcal X}^\otimes$. Of course, the same conclusion holds when $q_1=0$.

Finally, let us prove that $\widetilde{\mathcal M}_1^\otimes$ is dense in ${\widetilde{\mathcal X}^\otimes}$. 
If we start with $\xi^\otimes= (\widetilde{\alpha_j\beta_j})_j\in {\widetilde{\mathcal X}^\otimes}$, we can choose $n$ large enough such that both $\sum_{j>n} \alpha_j(\R^d)$ and 
$\sum_{j>n} \beta_j(\R^d)$ are negligible and work with a sequence of spatial points $a_1,\dots,a_n\in \R^d$ 
so that $\inf_{i\ne j} |a_i -a_j|\to \infty$, and take
the convex combinations
\begin{equation}\label{cvxcomb}
\begin{aligned}
&\mu_n=\sum_{j=1}^n \alpha_j \star \delta_{a_j}+ \bigg(1- \sum_{j=1}^n \alpha_j(\R^d)\bigg)\lambda_M \\
&\nu_n=\sum_{j=1}^n \beta_j \star \delta_{a_j}+ \bigg(1- \sum_{j=1}^n \beta_j(\R^d)\bigg)\lambda_M,
\end{aligned}
\end{equation}
where $\lambda_M$ is a centered Gaussian measure with variance $M\gg 1$, so that it totally disintegrates as $M\to \infty$ (recall \eqref{disint}). Then, repeating similar arguments as before, we see that 
as $M$ grows large with $n$, the sequence $\widetilde{\mu_n\nu_n}$ converges in
$\widetilde{\mathcal X}^\otimes$ to $\xi^\otimes=\{\widetilde{\alpha_j\beta_j}\}$. This concludes the proof of Theorem \ref{thmcompact}.
\end{proof}

Let us finally remark that if any sequence $\widetilde{\mu_n\nu_n}$ in $\widetilde\Mcal_1^\otimes$ converges to some $\xi^\otimes=(\widetilde{\alpha_j\beta_j})_j$, then
$$
\int\int V(x-y)\mu_n(\d x)\nu_n(\d y) \rightarrow \sum_j \int\int V(x-y) \alpha_j(\d x) \beta_j(\d y).
$$
for any continuous function $V$ vanishing at infinity. Since $\widetilde\Mcal_1^\otimes$ is dense in $\widetilde{\mathcal X}^\otimes$, we have the following
\begin{cor}\label{corcont}
For any $V\in \mathcal F_1^\otimes$, the functional $H: \widetilde{\mathcal X}^\otimes \rightarrow \R$ defined by
$$
H(\xi^\otimes)=\sum_j \int\int V(x-y) \alpha_j(\d x) \beta_j (\d y) \qquad \xi^\otimes=(\widetilde{\alpha_j\beta_j})_j
$$
is continuous.
\end{cor}


 \subsection{Large deviations in the compact space $\widetilde{\mathcal X}^\otimes$.}\label{sec-compactldp}
 
We recall that $\P^{\ssup 1}$ and $\P^{\ssup 2}$ denotes 
Wiener measures corresponding to two independent $d$ Brownian motions $W^{\ssup 1}$ and $W^{\ssup 2}$ starting from the origin, and 
$\P^\otimes=\P^{\ssup 1}\otimes \P^{\ssup 2}$. For any $i=1,2$, we also denote by,
$$
L^{\ssup i}_t= \frac 1t \int_0^t \d s \, \delta_{W^{\ssup i}_s} \in \Mcal_1
$$
the normalized occupation measure of th $i^{\mbox{th}}$ motion until time $t$. Since 
$$ 
L^\otimes_t= L^{\ssup 1}_t\otimes L^{\ssup 2}_t \in \Mcal_1^\otimes \to \widetilde\Mcal_1^\otimes \subset \widetilde{\mathcal X}^\otimes,
$$
and $\widetilde L^\otimes_t$ can be identified as an element in the space space $\widetilde{\mathcal X}^\otimes$ and we will prove a large deviation principle
for the distributions of $\widetilde L^\otimes_t$ under the product measure $\P^\otimes$.

Let us first introduce the classical Donsker-Varadhan rate function 
\begin{equation}\label{Idef}
I(\mu) = 
\begin{cases}
\frac 12 \|\nabla f\|_2^2
\qquad\quad\mbox{if} \,\,f=\sqrt{\frac{\d \mu}{\d x}}\in H^1(\R^d)
\\
\infty \qquad\qquad\quad\quad\mbox{else.}
\end{cases}
\end{equation}
Here $H^1(\R^d)$ is the usual Sobolev space of square integrable functions with square integrable derivatives. Note that the function $\mu\mapsto I(\mu)$ is translation invariant and depends only on the orbit $\widetilde\mu$.
Furthermore, this map is convex and homogenous of degree $1$. It is well-known (\cite{DV75-1}-\cite{DV83-4}) that 
the family of distributions of any occupation measure $L_t=1/t\int_0^t \delta_{W_s} \d s$ 
under any Wiener measure $\P$ satisfies a ``weak" large deviation principle in the space probability measures on $\Mcal_1(\R^d)$ with the rate function $I$. This means, under the weak topology, for every compact subset $K\subset\Mcal_1(\R^d)$
and for every open subset $G\subset\Mcal_1(\R^d)$,
\begin{equation}\label{DVub}
\limsup_{t\to\infty} \frac 1t \log \P(L_t\in K) \leq -\inf_{\mu\in K} I(\mu)
\end{equation}
\begin{equation}\label{DVlb}
\liminf_{t\to\infty} \frac 1t \log \P(L_t\in G) \geq -\inf_{\mu\in G} I(\mu),
\end{equation}
If for any family of distributions the upper bound \eqref{DVub} holds for any closed set, we say that the family satisfies a strong large deviation principle, or
just large deviation principle.

Let us also introduce the functional $\widetilde I: \widetilde{\mathcal X}^\otimes\to [0,\infty]$ given by
\begin{equation}\label{Itilde}
{\widetilde I}\big(\xi^\otimes\big)=\sum_{j} \big[I(\alpha)+I(\beta)\big] \qquad \xi=\{\widetilde{\alpha_j\beta_j}\}_j
\end{equation}
where $I$ is defined in \eqref{Idef} and $\alpha_j, \beta_j \in \Mcal_{\leq 1}$ so that the product $\alpha_j\beta_j$ is any arbitrary element of the orbit $\widetilde{\alpha_j\beta_j}$. Also recall that, $I(\cdot)$ is translation invariant.
\begin{theorem}\label{thmldp}
The distributions of $\widetilde L_t^\otimes$ under $\P^\otimes$ satisfies a large deviation principle in the compact metric space $\widetilde{\mathcal X}^\otimes$ with rate function $\mathfrak I$.
\end{theorem}
\begin{proof}
{\bf{The upper bound.}} We need to show that, for any closed set $F\subset\widetilde{\mathcal X}^\otimes$ in the metric $\mathbf D^\otimes$
$$
\limsup_{t\to\infty} \frac 1t \log \P^\otimes\big[\widetilde L_t^\otimes\in F\big] \leq - \inf_{\xi^\otimes \in F} \widetilde I(\xi^\otimes).
$$
Since $\widetilde{\mathcal X}^\otimes$ is a compact metric space, it suffices to show that, for any $\xi^\otimes= (\widetilde{\alpha_j\beta_j})_j\in \widetilde{\mathcal X}^\otimes$ and any 
neighborhood $U(\xi)$ of $\xi^\otimes$,
\begin{equation}\label{Qtub}
\begin{aligned}
\limsup_{t\to\infty} \frac 1t \log \P^\otimes\big[\widetilde L_t^\otimes\in U(\xi)\big] 
&=\limsup_{t\to\infty} \frac 1t \log \P^\otimes\big[\,\,\widetilde{L^{\ssup 1}_tL_t^{\ssup 2}} \,\,\in \,\,U\big(\,\{\widetilde{\alpha_j\beta_j}\}_j\,\big)\big] \\
&\leq - \sum_j \big[I(\alpha_j)+I(\beta_j)\big]
\end{aligned}
\end{equation}
Let us recall the convergence criterion in the space $\widetilde{\mathcal X}^\otimes$ defined via the decomposition \eqref{proddecomp}. 
Then the requirements \eqref{eq4thm}-\eqref{eq7thm} imply that to estimate the left hand side in \eqref{Qtub}, it is enough to upper bound the probability
\begin{equation}\label{ub0}
\begin{aligned}
\P^\otimes\bigg\{\exists c_1, \dots, c_k\in \R^d, \, &k\in \N, r>0\colon \,\, |c_i- c_j| \geq 10r \,\,\forall i\ne j, \\
&L^{\ssup 1}_t\big|_{B_r(0)} \in U_1\big(\alpha_j\star\delta_{c_j}\big), \, L^{\ssup 2}_t\big|_{B_r(0)} \in U_2\big(\beta_j\star\delta_{c_j}\big)\bigg\},
\end{aligned}
\end{equation}
where $U_1, U_2$ again denote neighborhoods in the weak topology in $\Mcal_{\leq 1}$. The last statement requires a remark:  Loosely speaking, 
note that the requirement \eqref{eq5thm} and \eqref{eq2thm}-\eqref{eq3thm} imply that, for the event 
$$
\bigg\{\widetilde{L^{\ssup 1}_tL_t^{\ssup 2}} \approx \{\widetilde{\alpha_j\beta_j}\}_j\,\,\mbox{ in }\,\, \widetilde{\mathcal X}^\otimes\bigg\}
$$
to occur, $L^{\ssup 1}_t$ and $L_t^{\ssup 2}$ both must find
``common concentration regions" $B_r(c_1),\dots B_r(c_k)$ where they resemble $\alpha_1,\dots \alpha_k$ and $\beta_1\dots,\beta_k$ respectively,
in the usual weak topology in $\Mcal_\leq 1$.

Again by independence, the probability in \eqref{ub0} splits into the product
$$
\begin{aligned}
&\P^{\ssup 1}\bigg\{\exists c_1, \dots, c_k\in \R^d, \, k\in \N, r>0\colon \,\, |c_i- c_j| \geq 4r, \,\,
L^{\ssup 1}_t\big|_{B_r} \in U_w\big(\alpha_j\star\delta_{c_j}\big)\bigg\}\\
&\times \P^{\ssup 2}\bigg\{\exists c_1, \dots, c_k\in \R^d, \, k\in \N, r>0\colon \,\, |c_i- c_j| \geq 4r, \,\,\, L^{\ssup 2}_t\big|_{B_r} \in U_w\big(\beta_j\star\delta_{c_j}\big)\bigg\}.
\end{aligned}
$$
Proposition 4.4 in \cite{MV14} identifies $\sum_j I(\alpha_j)$ and $\sum_j I(\beta_j)$ as the exponential decay rates of the probabilities above and finishes the proof of \eqref{Qtub}.
\medskip\noindent

{\bf{The Lower bound.}} It suffices to show that for 
any $\xi^\otimes\in \widetilde{\mathcal X}^\otimes$ with $\mathfrak I^\otimes(\xi^\otimes)<\infty$ and any neighborhood $U$ of $\xi^\otimes$,
\begin{equation}\label{lb}
\liminf_{t\to\infty} \frac 1t \log \mathbb \P^\otimes(\widetilde L_t^\otimes \in U) \geq - \widetilde I (\xi^\otimes)
\end{equation}
Let us fix any $\xi^\otimes= \{\widetilde{\alpha_j\beta_j}\}_j\in \widetilde{\mathcal X}^\otimes$ with $\mathfrak I^\otimes(\xi^\otimes)<\infty$. By the 
proof of Theorem \ref{thmcompact}, the space $\widetilde\Mcal_1^\otimes$ is dense in $\widetilde{\mathcal X}^\otimes$
and there is a sequence $\widetilde{\mu_n\nu_n}$ with $\mu_n,\nu_n\in \Mcal_1$ which converges to $\xi^\otimes$, recall the convex decomposition of $\mu_n$ and $\nu_n$ constructed in \eqref{cvxcomb}.
Furthermore, since $I(\cdot)$ is convex and $1$-homogeneous, such a convex decomposition implies that
\begin{equation}\label{claim1}
\begin{aligned}
 \limsup_{n\to\infty} \big[ I(\mu_n)+ I(\nu_n)\big] 
\leq \sum_j \big[I(\alpha_j\star\delta_{a_j})+I(\beta_j\star\delta_{b_j})\big]
=\sum_j \big[I(\alpha_j)+I(\beta_j)\big] 
=\widetilde I(\xi^\otimes).
\end{aligned}
\end{equation}
Since $U$ is any neighborhood of $\xi^\otimes$ in the space $\widetilde{\mathcal X}^\otimes$ and $\widetilde{\mu_n\nu_n} \to \xi^\otimes$ in this space with $\mu_n,\nu_n\in \Mcal_1$,
for $n$ large enough,
$$
\begin{aligned}
\P^\otimes\big[\widetilde L_t^\otimes\in U\big] &\geq \P^\ssup 1\big[\widetilde L^{\ssup 1}_t \in U_1(\widetilde \mu_n)\big] \times \P^\ssup 2\big[\widetilde L^{\ssup 2}_t \in U_2(\widetilde \nu_n)\big]  \\
&=\P^\ssup 1\big[ L^{\ssup 1}_t \in U_1( \mu_n)\big] \times \P^\ssup 2\big[ L^{\ssup 2}_t \in U_2( \nu_n)\big] 
\end{aligned}
$$
where $U_1(\mu_n), U_2(\nu_n)$ denote some neighborhoods of $\mu_n$ and $\nu_n$ in the usual weak topology in $\Mcal_1$.
Then by the classical lower bound \eqref{DVlb},
$$
\liminf_{t\to\infty} \frac 1t \log \P^\otimes\big[\widetilde L_t^\otimes\in U\big] \geq \liminf_{n\to\infty} \bigg\{-I(\mu_n)- I(\nu_n)\bigg\} \geq - \widetilde I(\xi^\otimes).
$$
Note that the last bound follows from \eqref{claim1}. This finishes the proof of the lower bound and Theorem \ref{thmldp} is also proved.
\end{proof}

\section{Dirac type mutual interaction: Brownian intersection measures in $\R^3$.}\label{sec-BrowInt}

We now turn to our first application to a model introduced in Section \ref{intro-sec-Dirac}.
The Gibbs measure we will be interested in is written formally as 
\begin{equation}\label{GibbsDelta}
\frac 1 {Z_t} \exp\bigg\{\frac 1t \int_0^t \int_0^t \d\sigma\, \d s\,\,\delta_0\big(\omega^{\ssup 1}_\sigma- \omega^{\ssup 2}_s\big)\bigg\} \d\P^\otimes
\end{equation}
Here $\delta_0$ denotes the usual Dirac measure, while $Z_t$ is the normalizing constant, and $\P^{\ssup 1}$ and 
$\P^{\ssup 2}$ are Wiener measures corresponding to two independent Brownian motions $W^{\ssup 1}$ and $W^{\ssup 2}$ in $\R^3$
starting at the origin and $\P^\otimes=\P^{\ssup 1}\otimes \P^{\ssup 2}$.
Recall that the exponential weight appearing in the above measures can be written (dropping the $1/t$ factor) as
the density (at zero) of the measure in $\R^3$, defined formally as
\begin{equation}\label{symbolical1}
\ell_t(A)= \int_A \,\d y\hspace{1mm} \prod_{i=1}^2 \int_0^{t}\d s \, \delta_y(W^{\ssup i}_s) \qquad A\subset\R^3
\end{equation}
supported on the non-empty random set 
\begin{equation*}
S=\bigcap_{i=1}^2 W^{\ssup i}_{[0,t]}= \bigg\{x\in \R^3\colon\, x=W^{\ssup 1}(\sigma)= W^{\ssup 2}(s)\,\,\,\mbox{ for some } \sigma, s\in [0,t]\bigg\}
\end{equation*}
of intersections of $W^{\ssup 1}$ and $W^{\ssup 2}$ until time $t$. We also recall that the expression \eqref{symbolical1} is also formal in $\R^3$ due to
the singularity of the occupation measures $L^{\ssup i}_t$ in $\R^3$. We refer to \cite{GHR84}, \cite{LG86} for a definition and a rigorous construction of $\ell_t$. 
As mentioned before, tail behavior of the total mass $\{\ell_t(\R^d)> at^2\}$ for $a>0$ and as $t\to\infty$ have been studied (\cite{KM01}, \cite{Ch03}),
and as measures, the distributions of $t^{-2}\ell_t$ under $\P^\otimes$ have been analyzed in \cite{KM13} based on classical theory of Donsker and Varadhan, in terms of a (weak) large deviation principle
with assumption that the motions do not leave a given bounded set until a large time $t$. The main goal of this section is to 
drop any assumption on restricting paths to a bounded domain and derive the asymptotic behavior of the distributions $t^{-2}\ell_t$ under the Gibbs transformation of the form \eqref{GibbsDelta}. Note that the model \eqref{GibbsDelta} enjoys an inherent ``diagonal shift-invariance" property which dictates $\delta_0(W^{\ssup 1}_\sigma- W^{\ssup 2}_s)=\delta_x(W^{\ssup 1}_\sigma- W^{\ssup 2}_s)$ for any $x\in\R^3$,
and it is conceivable that the desired analysis of the measures \eqref{GibbsDelta} yields to the theory developed in Section \ref{sec-paircompact}.

Let us quickly describe the organization of the present section. In Section \ref{sec-3.1} we declare the statements of our main results (Theorem \ref{thm-LDPBrowInt} and Theorem \ref{thm2-LDPBrowInt}), prove
Theorem \ref{thm2-LDPBrowInt} and end with a sketch of the proof of Theorem \ref{thm-LDPBrowInt}, which is carried out in three main steps. In Section \ref{sec-mollify}- Section \ref{sec-Gamma}, we carry out these three steps in detail. 



\subsection{Results: Convergence of path measures under Dirac type interaction.}\label{sec-3.1}

Let us denote by $\Mcal=\Mcal(\R^3)$ the space of finite measures in $\R^3$, equipped
with the vague topology induced by convergence of test integrals against continuous functions with 
compact support. As usual, we denote by $\widetilde\Mcal=\{\gamma\star\delta_x\colon\, \gamma\in\Mcal,\, x\in \R^3\}$
the quotient space of orbits $\widetilde\gamma\in\widetilde\Mcal$. 
We define the space
\begin{equation}\label{spaceY}
\widetilde{\Mcal}^\N= \big\{\{\widetilde\gamma_j\}_{j\in J} \colon  \,\gamma_j\in \Mcal\,\,\forall j\in J\big\}
\end{equation}
of empty, finite or countable collections $\xi= \{\widetilde\gamma_j\}_j$, which inherits the product topology
from the quotient space $\widetilde\Mcal$. We remark that, the Brownian intersection measure $\ell_t$ in $\R^3$ is random element of $\Mcal$
and our first main result concerns an abstract large deviation principle for the distributions of $t^{-2} \widetilde\ell_t$ in the space $\widetilde{\Mcal}^\N$. 
\begin{theorem}\label{thm-LDPBrowInt}
The family of distributions of $\big(t^{-2}\widetilde\ell_t\big)_{t\geq 0}$ under $\P^\otimes$ in the space $\widetilde{\Mcal}^\N$ satisfies a large deviation principle 
with rate function 
\begin{equation}\label{J}
\begin{aligned}
\mathfrak J(\xi)= \inf \bigg\{\frac 12 \sum_j\bigg( \|\nabla\psi_j\|_2^2+\|\nabla\phi_j\|_2^2\bigg)\colon \,\,\frac{\d\gamma_j}{\d x}&=\psi_j^2\,\phi_j^2\,\,\mbox{ for }\,\, \psi_j,\,\phi_j\in H^1(\R^3)\,\,  \forall j\,\\
 &\mbox { and }\,\,\sum_j \|\psi_j\|_2^2\leq 1, \, \sum_j \|\phi_j\|_2^2\leq 1\bigg\}
\end{aligned}
\end{equation}
for $\xi=\{\widetilde\gamma_j\}_j\in \widetilde{\Mcal}^\N$ such that for each $j$, $\gamma_j\in \Mcal(\R^3)$ has a density, and $\mathfrak J(\xi)=\infty$ otherwise. 
Equivalently, for any continuous and bounded function $\widetilde H:\widetilde{\Mcal}^\N\to \R$, 
$$
\lim_{t\to\infty} \frac 1t \log \E^\otimes\bigg\{\exp\big\{t \widetilde H\big(t^{-2}\widetilde \ell_t\big)\big\}\bigg\}= \sup_{\xi\in\widetilde{\Mcal}^\N}\big\{\widetilde H(\xi)- \mathfrak J(\xi)\big\}.
$$
\end{theorem}

One particularly interesting choice of such a shift-invariant functional $\widetilde H: \widetilde{\Mcal}^\N\to \R$ is given by 
\begin{equation}\label{eq-totalmass}
\widetilde H(\xi)= \sum_j\int_{\R^3} \mathbf{1} \,\gamma_j(\d x)= \sum_j\gamma_j(\R^3) \qquad\xi=\{\widetilde\gamma_j\}.
\end{equation}
This choice leads to the following assertion.  
\begin{theorem}\label{thm2-LDPBrowInt}
Let 
\begin{equation}\label{Qdef}
\d\mathbb Q_t=\frac 1 {Z_t} \exp\big\{t^{-1} \ell_t(\R^3)\big\} \d\P^\otimes
\end{equation}
where $Z_t=\E^\otimes\big[\exp\{t^{-1} \ell_t(\R^3)\big]$. Then 
\begin{equation}\label{eq-cor}
\limsup_{t\to\infty}\frac 1t\log \mathbb Q_t\big\{t^{-2}\ell_t \notin U(\mathfrak m)\big\} <0,
\end{equation}
where $U(\mathfrak m)$ denotes any neighborhood of the set
$$
\mathfrak m= \big\{\psi_0^4\star\delta_x\colon x \in \R^3\big\},
$$
in the vague topology and $\psi_0\in H^1(\R^3)$ is a smooth, rotationally symmetric function vanishing at infinity with $\|\psi_0\|_2=1$ and is the unique (modulo spatial shifts) maximizer 
of the translation invariant variational problem
$$
\chi=\sup_{\|\psi\|_2=1}\bigg\{\int_{\R^3} \d x \, \psi^4(x) \,\,-  \|\nabla\psi\|_2^2\bigg\}.
$$
\end{theorem}
\begin{proof}
Note that, by Theorem \ref{thm-LDPBrowInt} and \eqref{eq-totalmass}, it follows that 
$$
\begin{aligned}
\lim_{t\to \infty}\frac 1t \log Z_t 
&=\sup_{\xi=\{\widetilde\gamma_j\}_j\in\widetilde{\Mcal}^\N} \,\,  \bigg\{\sum_j\bigg(\int_{\R^3} \d x \, \psi_j^2(x)\phi_j^2(x) \, - \frac 12 \|\nabla\psi_j\|_2^2-\frac 12 \|\nabla\phi_j\|_2^2\bigg)\colon \\
&\qquad\qquad\qquad\quad\,\,\, \gamma_j(\d x)= \psi_j^2(x)\phi_j^2(x) \d x, \,\, \psi_j, \phi_j \in H^1(\R^3)\,\,\forall j,\\
& \hspace{45mm} \,\, \sum_j \|\psi_j\|_2^2\leq 1,  \sum_j \|\phi_j\|_2^2\leq 1\bigg\} \\
&=\sup_{\heap{\psi_j \in H^1(\R^3)}{\sum_j \|\psi_j\|_2^2\,\,\leq 1}}  \,\,\,\,\sum_j\bigg\{\int_{\R^3} \d x \, \psi_j^4(x) \, -  \|\nabla\psi_j\|_2^2\bigg\},
\end{aligned}
$$
where the last identity follows from the bound $2 \psi^2 \phi^2 \leq \psi^4+\phi^4$. We will show that the last supremum is achieved when the collection $\{\psi_j\}_j$ consists of only one function $\psi$
with total mass $\|\psi\|_2=1$. This follows from the following strict super-additivity: If we denote by
$$
\chi(m)= \sup_{\int_{\R^3} \psi^2=m} \bigg\{\int_{\R^3} \psi^4(x) \, \d x- \|\nabla\psi\|_2^2\bigg\}>0,
$$
then we will show that $\chi(m_1+m_2)>\chi(m_1)+\chi(m_2)$. Indeed, for any $\psi$ with $\|\psi\|_2=1$, let us invoke the rescaling $\psi_m(x):= m^2 \psi(mx)$, so that $\int_{\R^3} \psi_m^2(x) \,\d x=m$, and 
$$
\int_{\R^3} \psi_m^4(x) \, \d x- \int_{\R^3} \d x \,\,\big|\nabla\psi_m(x)\big|^2= m^5 \int_{\R^3} \d x \,\, \psi^4(x) \, \d x- m^3 \int_{\R^3} \d x \,\,\big|\nabla\psi_(x)\big|^2
$$
Then,
$$
\begin{aligned}
\chi(m_1+m_2)&=(m_1+m_2)^5 \bigg(\int_{\R^3} \d x \,\, \psi^4(x) \, \d x- \frac 1{(m_1+m_2)^2} \int_{\R^3} \d x \,\,\big|\nabla\psi_(x)\big|^2\bigg) \\
&>\sum_{i=1}^2m_i^5\bigg(\int_{\R^3} \d x \,\, \psi^4(x) \, \d x- \frac 1{(m_1\vee m_2)^2} \int_{\R^3} \d x \,\,\big|\nabla\psi_(x)\big|^2\bigg)\\
&\geq \sum_{i=1}^2m_i^5\bigg(\int_{\R^3} \d x \,\, \psi^4(x) \, \d x- \frac 1{m_i^2} \int_{\R^3} \d x \,\,\big|\nabla\psi_(x)\big|^2\bigg)\\
&=\chi(m_1)+\chi(m_2).
\end{aligned}
$$
This proves that
\begin{equation}\label{chidef}
\lim_{t\to\infty}\frac 1t\log Z_t= \sup_{\|\psi\|_2=1}\bigg\{\int_{\R^3} \d x \, \psi^4(x) \,\,-  \|\nabla\psi\|_2^2\bigg\}=\chi.
\end{equation}
It is well-known (\cite{T76}) that the above variational problem has a 
smooth, rotationally symmetric maximizer $\psi_0\in H^1(\R^3)$ which decays at infinity with $\|\psi_0\|_2=1$,
and $\psi_0$ is unique modulo translations in $\R^3$.

Now recall the definition of $\mathbb Q_t$ from \eqref{Qdef}. Then Theorem \ref{thm-LDPBrowInt} and \eqref{eq-totalmass} again imply that the distributions $\mathbb Q_t \circ \big(t^{-2} \widetilde{\ell}_t\big)^{-1}$ satisfy a large deviation principle in the space $\widetilde{\Mcal}^\N$ with rate function $\chi- \Theta(\cdot)$ where
$$
\Theta(\xi)= \sum_j\bigg(\int_{\R^3} \d x \, \psi_j^2(x)\phi_j^2(x) \, - \frac 12 \|\nabla\psi_j\|_2^2-\frac 12 \|\nabla\phi_j\|_2^2\bigg)
$$
if $\xi=\{\gamma_j\}_j$ and $\gamma_j(\d x)= \psi_j^2(x)\phi_j^2(x) \d x$ and $\sum_j \|\psi_j\|_2^2\leq 1$, $\sum_j \|\phi_j\|_2^2\leq 1$, else $\Theta(\xi)=\infty$.
In particular, for any closed set $\widetilde F\subset\widetilde{\Mcal}^\N$, 
$$
\limsup_{t\to\infty}\frac 1t \log \mathbb Q_t\big\{t^{-2}\widetilde\ell_t \,\in\,\widetilde F\big\} \leq \, -\inf_{\xi\in\widetilde F} \big[\chi- \Theta(\xi)\big].
$$
Let us denote by $\gamma_0(\d x)= \psi_0^4(x) \d x \in \Mcal(\R^3)$, where $\psi_0$ is the unique (modulo translation) maximizer of the variational problem
\eqref{chidef}, and $\{\widetilde\gamma_0\}\in \widetilde{\Mcal}\subset\widetilde\Mcal^\N$. 
If we now choose $\widetilde F$ to be the complement of any open neighborhood of $\{\widetilde\gamma_0\}$, then the uniqueness of
 $\psi_0$ implies that the above infimum is strictly positive and
$$
\mathbb Q_t \circ \big(t^{-2} \widetilde{\ell}_t\big)^{-1} \Rightarrow \delta_{\widetilde\gamma_0},
$$
proving the claim \eqref{eq-cor} and Theorem \ref{thm2-LDPBrowInt}.
\end{proof}

{\bf{Proof of Theorem \ref{thm-LDPBrowInt}:}} The proof of this theorem consists of three main steps carried out in Section \ref{sec-mollify}- Section \ref{sec-Gamma},
and the statement of the Theorem follows directly by combining Lemma \ref{epsLDP}, Lemma \ref{lemma-superexp} and Lemma \ref{gammalimit}.
\qed



\subsection{Mollification in the compactification $\widetilde{\mathcal X}^\otimes$.}\label{sec-mollify}
 Recall that $L^{\ssup 1}_t$ and $L^{\ssup 2}_t$ denote the normalized occupation measures of two independent  Brownian paths
observed until time $t$. In this step, we will smoothen each occupation measure with respect to some fixed mollifier $\varphi_\eps$, and take a pointwise
product of the mollified densities. Then the orbits of these densities satisfy a LDP in the space $\widetilde{\Mcal}^\N$ with an $\eps$-dependent
rate function. The precise statement can be found in Lemma \ref{epsLDP}.

Let $\varphi_\eps$ be a smooth mollifier, i.e., if $\varphi=\varphi_1$ is a non-negative, $\mathcal C^\infty$-function on $\R^3$ having compact support with
with $\int_{\mathbb R^3}\varphi(y)\,\d y=1$, we take
$$
\varphi_\eps(x)=\eps^{-d}\varphi(x/\eps).
$$
Then $\int_{\R^3} \varphi_\eps=1$ and $\varphi_\eps\Rightarrow \delta_0$ weakly as probability measures. 

Recall that for any $\alpha, \beta\in \Mcal_{\leq 1}$, we denote by $\alpha\beta=\alpha\otimes\beta\in \Mcal_{\leq 1}^\otimes$ the product measure. 
and by
$$
\alpha_\eps(x)= \big(\alpha\star\varphi_\eps\big)(x), \qquad \beta_\eps(x)=(\beta\star\varphi_\eps)(x),
$$
the mollified densities of $\alpha$ and $\beta$. Let us recall the compact metric space
$$
\widetilde{\mathcal X}^\otimes= \bigg\{\big\{\widetilde{\alpha_j\beta_j}\big\}_j\colon\,\,\alpha_j,\beta_j\in\Mcal_{\leq 1}, \sum_j\alpha_j(\R^3)\leq 1, \sum_j\beta_j(\R^3)\leq 1\bigg\},
$$ 
defined in \eqref{Xotimes}.
Then we consider the mapping
\begin{equation}\label{contmap}
\begin{aligned}
&\widetilde{\mathcal X}^\otimes \longrightarrow \widetilde{\Mcal}^\N \\
&\big\{\widetilde{\alpha_j\beta_j}\big\}_j \mapsto \big\{\widetilde\gamma_{j,\eps}\big\}_j,
\end{aligned}
\end{equation}
where for each $j$ and fixed $\eps>0$, we write 
$$
\gamma_{j,\eps}(\d x)= \alpha_{j,\eps}(x)\, \beta_{j,\eps}(x)\,\, \d x \in \Mcal(\R^3).
$$
We claim that this mapping is continuous. Since the space $\widetilde\Mcal_{\leq 1}^\otimes$ is dense in $\widetilde{\mathcal X}^\otimes$ (recall Theorem \ref{thmcompact}), 
for this continuity assertion it suffices to check that, for any $\alpha,\beta\in \Mcal_{\leq 1}$, the map
$$
\widetilde\Mcal_{\leq 1}^\otimes\,\,\ni \widetilde{\alpha\beta} \mapsto\widetilde\gamma_\eps \,\,\in \widetilde\Mcal\,\, \hookrightarrow\widetilde{\Mcal}^\N
$$ 
is weakly continuous, where $\gamma_\eps(\d x)=\alpha_\eps(x)\beta_\eps(x) \, \d x$. 
Indeed, for every continuous bounded test function $f\colon \R^3\to\R$ and any $\alpha, \beta\in\Mcal_{\leq 1}$, we have
\begin{equation*}
\begin{aligned}
\big\langle f,\gamma_\eps\big\rangle=\big\langle f, \alpha_\eps\beta_\eps\big\rangle
&=\int_{\R^d}\,\d x f(x) \int_{(\R^3)^2}
\alpha(\,\d y_1)\,\beta(\,\d y_2) 
\hspace{1mm}\varphi_\eps(x-y_1)\,\varphi_\eps(x-y_2)\\
&=\big\langle A_f,\alpha\otimes\beta\big\rangle,
\end{aligned}
\end{equation*}
where
$$
A_f(y_1,y_2)= \int_{\mathbb{R}^3}\d x\, f(x)
\hspace{1mm}\varphi_\eps(x-y_1)\,\varphi_\eps(x-y_2),
$$
and as usual, for any function $g$ and any measure $\mu$ we denoted by $\langle g,\mu\rangle$ the integral $\int g \d\mu$. For each fixed $\eps>0$, since $\varphi_\epsilon$ is smooth and compactly supported in $\R^3$, the function $A_f$ is continuous and bounded in $\mathbb{R}^{6}$. This shows the continuity of the map in \eqref{contmap}. 

Let us consider the mollified occupation densities 
\begin{equation*}
L^{\ssup i}_{\eps,t}(y)=\varphi_\eps\star L^{\ssup i}_t(y)=\frac 1t \int_0^t\d s\hspace{1mm}\varphi_{\eps}(W^{\ssup i}_s-y) \qquad i=1,2.
\end{equation*}
For each fixed $\eps>0$, these are smooth bounded functions in $\R^3$.
We also denote by 
$$
L^\otimes_{\eps,t}=L^{\ssup 1}_{\eps,t}\otimes L^{\ssup 2}_{\eps,t} \,\,\in \,\,\Mcal_{1} ^\otimes
$$ 
and as usual, we identify its orbit 
$\widetilde L^\otimes_{\eps,t}$ an an element in $\widetilde{\mathcal X}^\otimes$. 
Let us consider 
\begin{equation}\label{ellepst}
t^{-2}\ell_{\eps,t}(y)= \prod_{i=1}^2 L^{\ssup i}_{\eps,t}(y).
\end{equation}
the point-wise product of the mollified occupation densities. Then $t^{-2} \ell_{\eps,t}(y) \d y \in \Mcal$,
and its orbit $t^{-2}\widetilde\ell_{\eps,t} \in \widetilde{\Mcal}^\N$ is the continuous image of $\widetilde L^\otimes_{\eps,t}$ 
under the map \eqref{contmap} in the space $\widetilde{\Mcal}^\N$. By Theorem \ref{thmldp}, the distributions of $\widetilde L^\otimes_{\eps,t}$ in the space $\widetilde{\mathcal X}^\otimes$ satisfy a
large deviation principle, and hence, by contraction principle, we have proved the following lemma:

\begin{lemma}\label{epsLDP}
For each fixed $\eps>0$,  the distributions of $\big(t^{-2}\widetilde{\ell_{\eps,t}}\big)_{t\geq 0}$ in the space $\widetilde{\Mcal}^\N$ satisfy a LDP with rate function
\begin{equation}\label{J_eps}
\begin{aligned}
\mathfrak J_\eps\big(\xi)=
\inf\Big\{\frac{1}{2}\sum_{j}\bigg(\|\nabla\psi_j\|_2^2+\|\nabla\phi_j\|_2^2\bigg)\colon\,\,\frac{\d\gamma_j}{\d x}&=\psi_{j, \eps}^2\,\phi_{j,\eps}^2\,\,\mbox{ for }\,\, \psi_j,\,\phi_j\in H^1(\R^3)\,\,  \forall j\,\mbox { and }\\
 &\sum_j \|\psi_j\|_2^2\leq 1, \, \sum_j \|\phi_j\|_2^2\leq 1\bigg\} ,
\end{aligned}
\end{equation}
for any $\xi=\{\widetilde \gamma_j\}_j\in \widetilde{\Mcal}^\N$ such that for each $j$, $\gamma_j$ has a density, while $\mathfrak J_\eps(\xi)=\infty$ otherwise, and in the variational formula \eqref{J_eps}, we wrote $\psi^2_{j,\eps}=\psi_j^2\star\varphi_\eps$
and $\phi_{j,\eps}^2= \phi_j^2 \star \varphi_\eps$.
\end{lemma}

\subsection{A super exponential estimate.}\label{sec-BrowInt-proof}

 We remark that $t^{-2}\ell_{\eps,t}$ for any fixed $t$, is an approximation of the 
(rescaled) intersection measure $t^{-2}\ell_t$ as $\eps\downarrow 0$. 
In this step, we will go much further and show that they are an ``exponentially good approximation" of $t^{-2}\ell_t$.
This is the content of the next lemma. 

\begin{lemma}\label{lemma-superexp}
For any $a>0$, 
\begin{equation}\label{eq-superexp}
\limsup_{\eps\downarrow 0} \limsup_{t\uparrow\infty} \frac 1t \log \P^\otimes\big\{\d\big(t^{-2}\widetilde\ell_t, t^{-2} \widetilde\ell_{\eps,t}\big)>a\big\}=-\infty
\end{equation}
where $\d$ is a metric that induces the product topology on the space $\widetilde\Mcal^{\N}$.
\end{lemma}
\begin{proof}
We first note that $\widetilde\ell_t$ and $\widetilde\ell_{\eps,t}$ are elements of the single orbit space $\widetilde\Mcal$ of finite measures in $\R^3$,
equipped with the vague topology (with the quotient operation) on
$\widetilde\Mcal$ generated by test integral with respect to continuous functions with compact support. Hence, to derive
\eqref{eq-superexp}, it is enough to show that for any continuous
function $f: \R^3 \to \R$ with compact support,
\begin{equation}\label{eq1-superexp}
\limsup_{\eps\downarrow 0} \limsup_{t\uparrow\infty} \frac 1t \log \P^\otimes\big\{ t^{-2}\big|\big\langle f, \, \ell_t- \ell_{\eps,t}\big\rangle\big|>a\big\}=-\infty.
\end{equation}
Note that by Chebycheff's inequality, for any $k\in \N$, 
\begin{equation}\label{chebyshev}
\P^\otimes\big\{ t^{-2}\big|\big\langle f, \, \ell_t- \ell_{\eps,t}\big\rangle\big|>a\big\} \leq a^{-k} \,\, t^{-2k}\,\, \E^\otimes\bigg\{\big|\big\langle f, \, \ell_t- \ell_{\eps,t}\big\rangle\big|^k\bigg\},
\end{equation}
Actually, owing to technical reasons, we choose to work with an random time horizon $\tau$ instead of a fixed time horizon $t$,
and a simple scaling argument shows that required estimates on the moments $\E^\otimes[\langle f, \, \ell_t- \ell_{\eps,t}\rangle^k]$ follow from the bounds for the moments of $\langle f, \ell_\tau-\ell_{\eps,\tau}\rangle$. Indeed, let us denote by $\tau$ an exponential time with parameter $1$, which is independent of both $W^{\ssup 1}$ and $W^{\ssup 2}$. If $E$ denotes expectation with respect to $\tau$, we write $\overline\E=E\otimes \E^{\otimes}$. Let us note that, by Brownian scaling, for any constant $\theta>0$ and any $r>0$, 
$\E^\otimes[\langle \ell_{\theta r}, f\rangle^k]=\,\, \theta^{k/2} \,\,\E^\otimes[\langle \ell_{r}, f\rangle^k]$.
Then, 
$$
\begin{aligned}
\overline\E\bigg\{\langle f, \ell_\tau-\ell_{\eps,\tau}\rangle^k\bigg\}&= E\big[(\tau t^{-1})^{k/2}\big] \, \,\E^\otimes\bigg\{\big|\big\langle f, \, \ell_t- \ell_{\eps,t}\big\rangle\big|^k\bigg\}\\
&= t^{-k/2} \, \Gamma\bigg(1+\frac k 2\bigg) \,\,\E^\otimes\bigg\{\big|\big\langle f, \, \ell_t- \ell_{\eps,t}\big\rangle\big|^k\bigg\}
\end{aligned}
$$
If we combine this estimate with \eqref{chebyshev}, we get,
\begin{equation}\label{eq1.5-superexp}
\P^\otimes\big\{ t^{-2}\big|\big\langle f, \, \ell_t- \ell_{\eps,t}\big\rangle\big|>a\big\} \leq \bigg[a^{-k} \,\,\frac{t^{k/2}}{\Gamma\big(1+\frac k 2\big)}\bigg]\,\, \,t^{-2k}\,\,\, \overline\E\bigg\{\big|\big\langle f, \, \ell_\tau- \ell_{\eps,\tau}\big\rangle\big|^k\bigg\},
\end{equation}
We will show that,
\begin{equation}\label{eq2-superexp}
\lim_{\eps\downarrow 0} \limsup_{k\uparrow\infty} \frac 1k \log \bigg[\frac 1 {k!^2} \,\,\overline\E\bigg\{\big|\big\langle f, \, \ell_\tau- \ell_{\eps,\tau}\big\rangle\big|^k\bigg\}\bigg]=-\infty.
\end{equation}
In \eqref{eq1.5-superexp} if we choose $k=\lceil t\rceil$ and apply Stirlings's formula, then \eqref{eq1-superexp} follows from \eqref{eq2-superexp}. 

We remark that it suffices to prove \eqref{eq2-superexp} without the absolute value inside the expectation, 
since for $k\to\infty$ along even numbers, we can simply drop the absolute value in \eqref{eq2-superexp}, and when when $k$ is odd, we can use Jensen's inequality to go from  the power $k$ to $k+1$ and use that $((k+1)!^2)^{k/(k+1)}\leq k!^2 C^k$ for some $C\in(0,\infty)$ and all $k\in\N$.

Hence we owe the reader only the proof of \eqref{eq2-superexp} without the absolute value inside the expectation. 
For any $\lambda\in \R^3$, let $\widehat\varphi_\eps(\lambda)= \int_{\R^3} \d x \, \e^{\mathbf i \, \langle\lambda, x\rangle} \, \varphi_\eps(x)$ denote the Fourier coefficient of the
mollifier $\varphi_\eps$ so that $|\widehat\varphi_\eps(\lambda) - 1| \to 0$ as $\eps\to 0$.
Then by the Fourier inversion formula, 
$$
\langle\ell_{\eps,t}, f \rangle= C \int_{\R^3} \d \lambda \,\,\widehat f(\lambda)\, \widehat\varphi_\eps(\lambda)\,\,\int_0^t\int_0^t \d\sigma\,\d s \,\, \e^{\mathbf i \langle \lambda, W^{\ssup 1}_\sigma- W^{\ssup 2}_s\rangle}
$$
for some positive constant $C$. Hence,
$$
\begin{aligned}
\E^\otimes\Big[\big\langle f, \ell_\tau-\ell_{\eps,\tau}\big\rangle^k\Big] 
&= C^k\int_{(\R^3)^k} \bigg(\prod_{j=1}^k \d \lambda_j \, \widehat f(\lambda_j)\,\big[ 1- \widehat\varphi_\eps(\lambda_j)\big]\bigg)  \\
&\qquad\times\prod_{l=1}^2\bigg[\int_{[0,\tau]^k}\d s_1\dots \d s_k \,\,\E^{\ssup 1}\bigg\{ \e^{\mathbf i \sum_{j=1}^k \langle \lambda_j , W^{\ssup l}_{s_j}\rangle}\bigg\}\bigg] 
\end{aligned}
$$
If we abbreviate 
$$
\int_{[0,\tau]^k_\leq} \d s = \int_{0\leq s_1\leq\dots\leq s_k\leq \tau} \d s_1\dots\d s_k
$$ 
and for any Brownian path $W$, invoke time-ordering and Markov property, we get
$$
\begin{aligned}
&\int_{[0,\tau]^k}\d s_1\dots \d s_k \,\, \E\bigg\{ \e^{\mathbf i \sum_{j=1}^k \langle \lambda_j , W_{s_j}\rangle}\bigg\} \\
&= \sum_{\sigma \in \mathfrak S_k} \int_{[0,\tau]^k_{\leq}} \d s \,\,\E\bigg\{ \e^{\mathbf i \sum_{j=1}^k \langle \lambda_{\sigma(j)} , W_{s_j}\rangle}\bigg\} \\
&= \sum_{\sigma \in \mathfrak S_k} \int_{[0,\tau]^k_{\leq}} \d s \,\,\E\bigg\{ \exp\bigg\{\mathbf i \sum_{j=1}^k \bigg\langle \sum_{l=j}^k\lambda_{\sigma(l)} , W_{s_j}-W_{s_{j-1}}\bigg\rangle\bigg\}\bigg\} \\
&= \sum_{\sigma \in \mathfrak S_k} \int_{[0,\tau]^k_{\leq}} \d s \,\,\prod_{j=1}^k\exp\bigg\{-(s_j-s_{j-1})\bigg| \sum_{l=j}^k\lambda_{\sigma(l)}\bigg|^2 \bigg\},
\end{aligned}
$$
where $\mathfrak S_k$ denotes the permutation group of $\{1,\dots,k\}$. 
Recall that $E$ denotes expectation with respect to the exponential time $\tau$ and we write $\overline \E=E\otimes\,\E^\otimes$. Then
$$
\begin{aligned}
\overline \E\Big[\big\langle f, \ell_\tau&-\ell_{\eps,\tau}\big\rangle^k\Big]
=C^k \int_{(\R^3)^k} \bigg(\prod_{j=1}^k \d \lambda_j \, \widehat f(\lambda_j)\,\big[ 1- \widehat\varphi_\eps(\lambda_j)\big]\bigg) \\
&\qquad\qquad\times\bigg[\sum_{\sigma \in \mathfrak S_k} \int_0^\infty \d r \,\,\e^{-r} \,\,\int_{[0,r]^k_{\leq}} \d s 
\,\,\prod_{j=1}^k\exp\bigg\{-(s_j-s_{j-1})\bigg| \sum_{l=j}^k\lambda_{\sigma(l)}\bigg|^2 \bigg\}\bigg]^2 \\
&\qquad\qquad =C^k \int_{(\R^3)^k} \bigg(\prod_{j=1}^k \d \lambda_j \, \widehat f(\lambda_j)\,\big[ 1- \widehat\varphi_\eps(\lambda_j)\big]\bigg)  
\\&\qquad\qquad\qquad\qquad\times\bigg[\sum_{\sigma \in \mathfrak S_k} \prod_{j=1}^k \int_0^\infty \exp\bigg\{-r\bigg(1+\bigg| \sum_{l=j}^k\lambda_{\sigma(l)}\bigg|^2 \bigg)\bigg\}\bigg]^2 \\
&= C^k \int_{(\R^3)^k} \bigg(\prod_{j=1}^k \d \lambda_j \, \widehat f(\lambda_j)\,\big[ 1- \widehat\varphi_\eps(\lambda_j)\big]\bigg)  
\times\bigg[\sum_{\sigma \in \mathfrak S_k} \prod_{j=1}^k \frac 1{1+\big| \sum_{l=j}^k\lambda_{\sigma(l)}\big|^2}\bigg\}\bigg]^2 .
\end{aligned}
$$
Then applying Jensen's inequality to the sum $\sum_{\sigma\in \mathfrak S_k}$, we get
$$
\begin{aligned}
&\overline\E\Big[\big\langle f, \ell_\tau-\ell_{\eps,\tau}\big\rangle^k\Big] \\
&\leq  C^k k!^2\,\, \int_{(\R^3)^k} \prod_{j=1}^k \bigg[\d \lambda_j \, \bigg(\widehat f(\lambda_j)\,\big[ 1- \widehat\varphi_\eps(\lambda_j)\big]\bigg)  
 \,\,\bigg(\frac 1{1+\big| \sum_{l=j}^k\lambda_{l}\big|^2}\bigg)^2\bigg] \\
  &\leq   \widetilde C^k\,\, k!^2\,\, \int_{(\R^3)^k} \prod_{j=1}^k \bigg[\d \lambda_j \, \big[ 1- \widehat\varphi_\eps(\lambda_j- \lambda_{j-1})\big]
 \,\,\bigg(\frac 1{1+\big| \lambda_{j}\big|^2}\bigg)^2\bigg]  \end{aligned}
$$
Now the facts that that $|1-\widehat\varphi_\eps(\lambda)|\to 0$ as $\eps\to 0$ and 
\begin{equation}\label{integrability} 
\int_{\R^3} \frac{\d\lambda}{(1+|\lambda|^2)^{2}}<\infty,
\end{equation} 
can be exploited to check that
\begin{equation}\label{eq4-superexp}
\begin{aligned}
&\lim_{\eps\to 0}\limsup_{k\to \infty}\frac 1k \log \int_{(\R^3)^k} \prod_{j=1}^k \bigg[\d \lambda_j \, \,\big[ 1- \widehat\varphi_\eps(\lambda_j- \lambda_{j-1})\big].
 \,\,\bigg(\frac 1{1+\big| \lambda_{j}\big|^2}\bigg)^2\bigg]\\
& \qquad=-\infty.
\end{aligned}
 \end{equation}
 Indeed, we can decompose the spatial integral $\int_{(\R^3)^k}$ with $(\R^3)^k ={(I)_k} \cup {(II)_k}$ where
 $$
 (I)_k= \bigg\{(\lambda_1,\dots,\lambda_k)\in (\R^3)^k\colon \#\{1\leq j\leq k\colon |\lambda_j|\geq R\} \geq \eta k\bigg\},
 $$
 for some $\eta\in(0,1)$ and $R>1$. If we work with the mollifier $\varphi_\eps(x)=c \exp\{-\frac{|x|^2}{2\eps}\}$ supported on $B_\eps(0)$ so that $\widehat\varphi_\eps(\lambda)=c\exp\{-\eps^2|\lambda|^2/2\}$,
 then for any given $\delta>0$, we can choose $\lambda_0=\lambda_0(\eps)$ small enough so that $1-\widehat\varphi_\eps(\lambda)<\delta$ for $|\lambda|<\lambda_0$, while
 $\int_{B_R(0)^c} {\d\lambda}\, (1+|\lambda|^2)^{-2}<\delta$ for $R$ large enough, thanks to \eqref{integrability}. Then on the set $(I)_k$, we can ignore the terms $1-\widehat\varphi_\eps(\cdot)\leq 1$ and 
 take advantage of the fact that at least $\eta k$ of the $k$ integrals are taken outside the ball of radius $R$ around the origin and these integrals are therefore small, while
the other $(1-\eta)k$ integrals yield only some bounded exponential rate, i.e., 
$$
 \begin{aligned}
 \int_{A_k} \prod_{j=1}^k \,\, \frac{\d \lambda_j} {(1+| \lambda_{j}^2|)^{2}} \,\,\,
 & \leq \binom{k}{\eta k} \,\,\,\bigg(\int_{\R^3} \frac{\d \lambda} {(1+| \lambda|^2)^{2}}\bigg)^{(1-\eta)k}\,\,\, \bigg(\int_{B_R(0)^c} \frac{\d \lambda}{ (1+| \lambda|^2)^{2}}\bigg)^{\eta k}  \\
 &\leq C(\delta)^{\eta k}
 \end{aligned}
 $$
with $C(\delta)\to 0$ as $\delta\to 0$. On the complement $(II)_k$, we can also use that, for suitably chosen $\eta$, there are at least $\eta k$ indices $j\in\{1,\dots,k\}$ such that,
$1-\widehat\varphi_\eps(\lambda_j-\lambda_{j-1}) \leq \delta$ (for $\eps$ small enough) and deduce a similar estimate for the integral $\int_{(II)_k}$ as above. If we combine these two estimates, and send $\delta\to 0$, we end up with \eqref{eq4-superexp}. This proves \eqref{eq1-superexp} and thus Lemma \ref{lemma-superexp}.
 \end{proof}

\subsection{Gamma convergence in the compactification $\widetilde{\mathcal X}^\otimes$.}\label{sec-Gamma}

In this section, we conclude the proof of Theorem \ref{thm-LDPBrowInt}. Given Lemma \ref{epsLDP} and Lemma \ref{lemma-superexp}, 
the proof of Theorem \ref{thm-LDPBrowInt} is complete (see \cite{DZ98}, Theorem 4.2.16), provided
we establish the following identity, which resembles the notion of $\Gamma$-convergence. Its proof is based on Theorem \ref{thmcompact}.

\begin{lemma}\label{gammalimit}
For every $\xi\in \widetilde{\Mcal}^\N$,
\begin{equation}\label{gamma}
\sup_{\delta>0}\liminf_{\eps\downarrow 0}\inf_{U_\delta(\xi)}\hspace{1mm}\mathfrak J_\eps=\hspace{1mm}\mathfrak J(\xi),
\end{equation}
where $\mathfrak J$ and $\mathfrak J_\eps$ are the rate functions defined in \eqref{J} and \eqref{J_eps}, respectively, and $U_\delta(\xi)$ denotes a
ball of radius $\delta$ around $\xi$ in the space $\widetilde{\Mcal}^\N$.
\end{lemma}

\begin{Proof}{Proof.} 
We note that the fact that (L.H.S) $\leq$ (R.H.S) is easy. Let us fix $\xi=\{\widetilde\gamma_j\}_j\in \widetilde{\Mcal}^\N$ and any $\delta>0$. 
We can assume that the right hand side is finite (otherwise there is nothing to prove) so that, 
for each $j$ we have $\psi_j, \phi_j\in H^1(\R^3)$ with $\sum_j \|\psi_j\|_2^2\leq 1$ and $\sum_j \|\phi_j\|_2^2\leq 1$ and $\gamma_j(\d x)= \psi_j^2(x) \, \phi_j^2(x) \, \d x$. Then 
we can define $\lambda_j(\d x)= (\psi_j^2\star\varphi_\eps)(x)  (\phi_j^2\star\varphi_\eps)(x) \, \d x$ and choose $\eps>0$ small enough so that 
$\{\widetilde\lambda_j\}_j\in \widetilde{\Mcal}^\N$ lies in the ball $U_\delta(\xi)$. This shows that, for such $\eps>0$, 
$\mathfrak J_\eps(\{\widetilde\lambda_j\}_j) \leq \mathfrak J(\xi)$. 

Now we prove (L.H.S) $\geq$ (R.H.S): Let $\xi=\{\widetilde\gamma_j\}_j\in \widetilde{\Mcal}^\N$ be given, and we first assume that $\mathfrak J(\xi)$ is finite. Then, for any $\eps, \delta>0$, there exists 
$\xi^{\ssup{\eps,\delta}}= \{\widetilde\gamma_j^{\ssup{\eps,\delta}}\}_j \in U_\delta(\xi)$ so that 
$$
\inf_{U_\delta(\xi)} \mathfrak J_\eps \geq \mathfrak J_\eps(\xi^{\ssup{\eps,\delta}}) - \delta.
$$
By our assumption, the right hand side is finite. Hence, by definition of $\mathfrak J_\eps$, for each $j$, there exists $\psi_j^{\ssup{\eps,\delta}}, \phi_j^{\ssup{\eps,\delta}} \in H^1(\R^3)$ with 
$\sum_j\big\|\psi_j^{\ssup{\eps,\delta}}\big\|_2\leq 1$ and $\sum_j\big\|\phi_j^{\ssup{\eps,\delta}}\big\|_2\leq1$
such that
\begin{equation}\label{gamma1}
\frac{\d\gamma_j^{\ssup{\eps,\delta}}}{\d x}= \big[\big(\psi_j^{\ssup{\eps,\delta}}\big)^2 \star \varphi_\eps\big] \,\,\big[ \big(\phi_j^{\ssup{\eps,\delta}}\big)^2 \star \varphi_\eps\big]
\end{equation}
and
\begin{equation}\label{gamma2}
\mathfrak J_\eps(\xi^{\ssup{\eps,\delta}}) \geq \frac 12 \sum_j\bigg[\big\|\nabla\psi_j^{\ssup{\eps,\delta}}\big\|_2^2+\big\|\nabla\phi_j^{\ssup{\eps,\delta}}\big\|_2^2\bigg] - \delta
\end{equation}
First we want to let $\eps\to 0$ on  the right hand side of \eqref{gamma2}. This requires an argument based on
the proof of Theorem \ref{thmcompact}.

For any fixed $\delta>0$ and fixed index $j$,
for notational convenience, let us write $\psi_\eps=\psi_j^{\ssup{\eps,\delta}}, \phi_\eps=\phi_j^{\ssup{\eps,\delta}}$,
denote by $\mu_\eps$ and $\nu_\eps$ measures $\psi_\eps^2(x) \d x$ and $\phi_\eps^2(x) \d x$, respectively.
Then the sequence 
$$
\widetilde{\mu_\eps\nu_\eps}\in \widetilde {\mathcal X}^\otimes
$$
converges to some element $\{\widetilde{\alpha_l\beta_l}\}_l\in \widetilde {\mathcal X}^\otimes$ along some subsequence, since $\widetilde {\mathcal X}^\otimes$ is compact
(Theorem \ref{thmcompact}). Recall the convergence criterion in the space $\widetilde {\mathcal X}^\otimes$, set in \eqref{eq2thm}-\eqref{eq4thm}.
This implies that there exists $k\in \N$, such that we can write
$$
\mu_\eps= \sum_{l=1}^k \alpha_l^{\ssup\eps}+ \lambda^{\ssup\eps}_1\qquad\nu_{\eps}= \sum_{l=1}^k \beta_l^{\ssup\eps}+ \lambda^{\ssup\eps}_2
$$
such that for each $l=1,\dots,k$, there exist spatial shifts $(c_l^{\ssup\eps})_{\eps>0}\subset\R^3$ with $\lim_{\eps\to 0} \inf_{l\ne m} |c_l^{\ssup\eps}- c_m^{\ssup\eps}| \to \infty$ and, for all $l=1,\dots,k$
\begin{equation}\label{gamma3}
\alpha_l^{\ssup\eps} \star \delta_{c_l^{\ssup\eps}} \Rightarrow \alpha_l \quad\mbox{ and }\quad \beta_l^{\ssup\eps} \star \delta_{c_l^{\ssup\eps}} \Rightarrow \beta_l ,
\end{equation}
while, for any $V\in \mathcal F^\otimes_1$ (i.e., any continuous function vanishing at infinity), 
$$
\max\bigg\{\limsup_{\eps\to 0} \int_{\R^6} V(x-y) \lambda_1^{\ssup\eps}(\d x) \lambda_1^{\ssup\eps}(\d y), \,\,\limsup_{\eps\to 0} \int_{\R^6} V(x-y) \lambda_2^{\ssup\eps}(\d x) \lambda_2^{\ssup\eps}(\d y)\bigg\} \leq \eta.
$$
Let $\chi(\cdot)$ be a smooth cut-off function which is $1$ in the unit ball and vanishes smoothly outside a ball of radius $2$ and $0\le \chi(\cdot)\le 1$. Let $r_\eps= \frac 14\inf_{l\ne m} |c_l^{\ssup\eps}-c_m^{\ssup m}|$ 
so that $r_\eps\to \infty$ as $\eps\to 0$. This leads to a partition of unity:
\begin{equation}\label{gamma4}
 1=\sum_{l=1}^k\bigg\{\chi\bigg(\frac{x+c^{\ssup \eps}_l}{r_\eps}\bigg)\bigg\}^2+\bigg[1-\sum_{l=1}^k\bigg\{\chi\bigg(\frac{x+c^{\ssup \eps}_l}{r_\eps}\bigg)\bigg\}^2\bigg] 
 \end{equation}
Recall that for every $\eps>0$, $\mu_\eps$ and $\nu_\eps$ have densities $\psi_\eps^2$ and $\phi_\eps^2$. We will show that if $\frac{1}{2}\big[\int_{\R^3}|\nabla \psi_\eps|^2 \d x+\int_{\R^3}|\nabla \phi_\eps|^2 \d x\big]\le A$,
 then for all $l=1,\dots, k$, $\alpha_l$ and $\beta_l$  are all absolutely continuous with densities $\psi_l^2$ and $\phi_l^2$ respectively, and
 \begin{equation}\label{gamma5}
 \sum_{l=1}^k \frac{1}{2}\bigg[\int_{\R^3}|\nabla \psi_l |^2 \d x+\int_{\R^3}|\nabla \phi_l|^2 \d x\bigg]\le A.
 \end{equation}
We define, for any $ l=1,\dots, k,$
$$
 \begin{aligned}
 \psi_l^{\ssup \eps}(x)&=\psi_\eps(x)\chi\bigg(\frac{x+c^{\ssup \eps}_l}{r_\eps}\bigg).
 \end{aligned}
 $$
so that
 $$
\frac{1}{2}\int |\nabla\psi_l^{\ssup\eps}|^2 \d x= \frac{1}{2}\int \bigg|\nabla \psi_\eps(x)\chi\bigg(\frac{x+c^{\ssup \eps}_l}{r_\eps}\bigg)+\frac{1}{r_\eps}\psi_\eps(x)\big(\nabla \chi\big)\bigg(\frac{x+c^{\ssup \eps}_l}{r_\eps}\bigg)\bigg|^2\d x
 $$
If we now let $\eps\to 0$ so that $r_\eps\to\infty$, the requirement \eqref{gamma3} implies that $\big(\psi^{\ssup \eps}_l\big)^2(x)\d x\Rightarrow \alpha_l$ for $l=1,2,\ldots,k$.
Furthermore, on the right hand side of the last display, since $\chi$  and $\nabla\chi$ are uniformly bounded and the integrals $\int |\psi_\eps(x)|^2 \d x$ and $\int |\nabla \psi_\eps|^2 \d x $ are also bounded, 
only the first term in the integral counts.
Thus, for $r_\eps\to\infty$ large enough, \eqref{gamma4} implies that
 $$
 \begin{aligned}
 \sum_{l=1}^k \int \bigg|\nabla \psi_\eps(x)\chi\bigg(\frac{x+c^{\ssup \eps}_l}{r_\eps}\bigg)\bigg|^2
 &\le \frac{1}{2}\int |\nabla \psi_\eps(x)|^2\d x
&\leq A/2 <\infty,
\end{aligned}
 $$
Since, the map $\big(\psi^{\ssup \eps}_l\big)^2(x)\d x \mapsto \frac{1}{2}\int |\nabla\psi_l^{\ssup\eps}|^2 \d x$ is weakly lower semicontinuous,
it follows that any weak limit $\alpha_l$ of $(\psi^{\ssup \eps}_l)^2(x) \d x$ has a density $\psi_{l}$ and $\sum_{l=1}^k \frac{1}{2}\int |\nabla\psi_l|^2 \d x\le A/2$.
Of course, the same argument works for the measures $\nu_\eps(\d x)=\phi_ \eps^2(x)\d x$, and we have proved \eqref{gamma5}.

Let us return to \eqref{gamma1} and \eqref{gamma2}. We write, for each $j$
$\mu_j^{\ssup{\eps,\delta}}(\d x)=\big(\psi_j^{\ssup{\eps,\delta}}\big)^2(x)\d x$ and $\nu_j^{\ssup{\eps,\delta}}(\d x)=\big(\phi_j^{\ssup{\eps,\delta}}\big)^2(x)\d x$.
Then for every fixed $\delta>0$, again by the compactness of the space $\widetilde{\mathcal X}^\otimes$, as $\eps\to 0$, we have
\begin{equation}\label{gamma7}
\widetilde{\mathcal X}^\otimes\,\,\ni \,\,\big\{\widetilde{\mu_j^{\ssup{\eps,\delta}}\nu_j^{\ssup{\eps,\delta}}}\big\}_j \longrightarrow\{\widetilde{\alpha_j^{\ssup\delta}\beta_j^{\ssup\delta}}\}_j \in \widetilde{\mathcal X}^\otimes.
\end{equation}
Furthermore, we let $\eps\to 0$ on the right hand side of \eqref{gamma2} and by the arguments leading to \eqref{gamma5}, we have
$$
\liminf_{\eps\to 0} \,\, \inf_{U_\delta(\xi)} \,\mathfrak J \geq \frac 12 \sum_j\bigg[\big\|\nabla\psi_j^{\ssup{\delta}}\big\|_2^2+\big\|\nabla\phi_j^{\ssup{\delta}}\big\|_2^2\bigg] - \delta,
$$
where, for each $j$, $(\psi_j^{\ssup\delta})^2$ and $(\phi_j^{\ssup\delta})^2$ are the densities of $\alpha_j^{\ssup\delta}$ and $\beta_j^{\ssup\delta}$, respectively. 
We remark that the existence of the (subsequential) limit \eqref{gamma7} continues to hold 
for the mollified measures $\mu_j^{\ssup{\eps,\delta}}$ and $\nu_j^{\ssup{\eps,\delta}}$ 
with densities $(\psi_j^{\ssup{\eps,\delta}})^2 \star \varphi_\eps$
and $(\phi_j^{\ssup{\eps,\delta}})^2 \star \varphi_\eps$, respectively. 

Repeating now the same argument now for $\delta\to 0$, we have 
\begin{equation}\label{gamma8}
\big\{\widetilde{\alpha_j^{\ssup\delta}\beta_j^{\ssup\delta}}\big\}_j \rightarrow \{\widetilde{\alpha_j\beta_j}\}_j\in \widetilde{\mathcal X}^\otimes,
\end{equation}
so that $\alpha_j$ and $\beta_j$ have densities $\psi_j^2$ and $\phi_j^2$ with $\sum_j \| \psi_j\|_2^2\leq 1$, $\sum_j \|\phi_j\|_2^2\leq 1$, and we also have
$$
\sup_{\delta>0}\liminf_{\eps\to 0} \,\, \inf_{U_\delta(\xi)} \,\mathfrak J \geq \frac 12 \sum_j\bigg[\big\|\nabla\psi_j\big\|_2^2+\big\|\nabla\phi_j\big\|_2^2\bigg].
$$
Let us now recall \eqref{gamma1} and note that $\xi^{\ssup{\eps,\delta}}= \{\widetilde\gamma_j^{\ssup{\eps,\delta}}\}_j$ lies in the ball $U_\delta(\xi)$ 
around $\xi=\{\widetilde\gamma_j\}_j\in \widetilde{\Mcal}^\N$.
By \eqref{gamma7}, \eqref{gamma8} and the continuity of the map \eqref{contmap}, it follows that $\gamma_j$ has density $\psi_j^2 \, \phi_j^2$, and hence, the right hand side of the above display
must be $\mathfrak J(\xi)$, by definition of $\mathfrak J(\cdot)$, recall \eqref{J}.

Finally, if $\mathfrak J(\xi)=\infty$, we can assume that the left hand side of \eqref{gamma} is finite. Then an exact repitition of the above arguments leads to a contradiction.
This concludes the proof of the Lemma \ref{gammalimit}.
\end{Proof}
\qed

 \section{Mean-field path measures under mutual Coulomb interaction}\label{sec-pairpolaron}
 
 We now focus on the model introduced in Section \ref{intro-sec-bipolaron} defined by the Gibbs measure
 $$
\begin{aligned}
\d \widehat\P_t^\otimes&=  \frac 1 {Z_t} \exp\bigg\{\frac 1t \int_0^t \int_0^t \frac {\d\sigma\,\,\d s} {|\omega^{\ssup 1}_\sigma- \omega^{\ssup 2}_s|}\bigg\} \,\,\d\P^{\otimes} \\
&= \frac 1 {Z_t} \exp\big\{t H\big(L_t^\otimes\big)\big\} \,\,\d \P^{\otimes},
\end{aligned}
$$
where 
$$
H(\mu\otimes\nu)= \int\int \frac 1 {|x-y|} \mu(\d x) \nu(\d y) \qquad \mu,\nu\in \Mcal_1=\Mcal_1(\R^3),
$$
and $Z_t=\E^\otimes[\exp\big\{t H\big(L_t^\otimes\big)\big\}]$. Note that in $\R^3$, this expectation is finite.

In this section, we will prove asymptotic convergence of the distribution
$$
\mathbb S_t=  \widehat\P^\otimes_t \circ \big(\widetilde L_t^\otimes\big)^{-1}
$$
as $t\to \infty$. Let us denote by
\begin{equation}\label{rhodef}
\rho= \sup_{\heap{\psi\in H^1(\R^d}{\|\psi\|_2=1}} \bigg\{ \int\int_{\R^3\times\R^3} \frac{\psi^2(x) \psi_2^2 (y)}{|x-y|} \d x \d y - \|\nabla\psi\|_2^2\bigg\}
\end{equation}
It is well-known (\cite{L76}) that the above variational problem 
has a rotationally symmetric maximizer $\psi_0$ which is unique except for spatial translations. Let $\mu_0$
be the measure with the maximizing density $\psi_0^2$ and $\widetilde{\mu_0\mu_0} \in\widetilde\Mcal_1^\otimes\subset\widetilde{\mathcal X}^\otimes$
denotes the orbit of the product measure. Here is our next main result. 
\begin{theorem}\label{thmpairtube}
As $t\to\infty$,
$$
 \mathbb S_t \Rightarrow \delta_{\widetilde{\mu_0\mu_0}}
 $$
 weakly as probability measures. Consequently, if $\mathfrak m^\otimes=\big\{(\mu_0\otimes\mu_0)\star\delta_x\colon x\in \R^3\big\}$,
 $$
 \frac 1 t\log \widehat\P_t^\otimes\big\{L_t^\otimes \notin U(\mathfrak m^\otimes)\big\}<0,
 $$
 where $U(\mathfrak m^\otimes)$ denotes any neighborhood of $\mathfrak m$ in the usual weak topology. 
\end{theorem}
\begin{proof}
We will first show that the family of distributions $\mathbb S_t$ satisfies a large deviation principle in $\widetilde{\mathcal X}^\otimes$ with rate function
\begin{equation}\label{Jdef}
J(\xi^\otimes)= \widetilde\rho-\sum_j \bigg\{\int \frac{ \alpha_j(\d x)\beta_j(\d y)}{|x-y|}-I({\alpha}_j)- I(\beta_j)\bigg\} \quad \xi^\otimes=(\widetilde{\alpha_j\beta_j}),
\end{equation}
where $I$ denotes the classical Donsker-Vradhan rate function and  
\begin{equation}\label{rhotilde}
\widetilde\rho=\sup \sum_j \bigg\{\int_{\R^d}\int_{\R^d} \frac{\psi_j^2(x)\phi_j^2(y)}{|x-y|} \d x\d y-\frac{1}{2}\sum_j\big\|\nabla \psi_j\big\|_2^2-\frac{1}{2}\sum_j\big\|\nabla \phi_j\big\|_2^2\bigg\},
\end{equation}
and the above supremum is taken over those elements $\{\widetilde{\alpha_j\beta_j}\}_j$ in $\widetilde{\mathcal X}^\otimes$ such that $\alpha_j$ and $\beta_j$ have densities $\psi^2_j$ and $\phi^2_j$ such that $\sum_j \int_{\R^3} \psi_j^2(x) \d x\leq 1$ and $\sum_j \int_{\R^d} \phi_j^2(y) \d y\leq 1$.

For any fixed $\delta>0$, let $V(x)=V_\delta(x)=(\delta^2+|x|^2)^{-1/2} \in \mathcal F_1^\otimes$. Let us write, for any $A\subset\widetilde{\mathcal X}$,
\begin{equation}\label{rewrite}
\begin{aligned}
\mathbb S_t(A)
&= \frac 1 {Z_t}{\E^{\otimes}\bigg\{ \exp\bigg\{\frac 1t\int_0^t\int_0^t V_\delta(W^{\ssup 1}_\sigma-W^{\ssup 2}_s) \d\sigma \d s \bigg\} \, \1_A\bigg\}} \\
&=\frac 1 {Z_t} \E^{\otimes}\bigg\{ \exp\big\{t H_\delta(\widetilde L^\otimes_t)\big\}\1_A\bigg\},
\end{aligned}
\end{equation}
where 
$$
H_\delta(\widetilde{\mu\nu})=\int\int_{\R^3\times\R^3} V_\delta(x-y) \mu(\d x) \nu(\d y) \qquad\mbox{ for any } \mu\nu\in\widetilde{\mu\nu}\in \widetilde{\mathcal X}^\otimes.
$$
For any fixed $\delta>0$, $H_\delta$ is a continuous function on the compact metric space $\widetilde{\mathcal X}^\otimes$ (Corollary \ref{corcont}), and
we can take $A=F\subset\widetilde{\mathcal X}^\otimes$ to be a closed set and invoke the large deviation upper bound in Lemma \ref{thmldp} and Varadhan's lemma to deduce,
$$
\begin{aligned}
&\limsup_{t\to\infty}\frac 1t \log \E^{\otimes}\bigg\{ \exp\big\{t H_\delta(\widetilde L^\otimes_t)\big\}\1_F\bigg\}\\
&\leq \sup_{\xi^\otimes\in F} \sum_j \bigg\{\int_{\R^d}\int_{\R^d} V_\delta(x-y)\psi_j^2(x)\phi_j^2(y) \d x\d y-\frac{1}{2}\sum_j\big\|\nabla \psi_j\big\|_2^2-\frac{1}{2}\sum_j\big\|\nabla \phi_j\big\|_2^2\bigg\}
\end{aligned}
$$ 
where $\xi^\otimes=(\widetilde{\alpha_j\beta_j})$ and $\alpha_j$ and $\beta_j$ have densities $\psi^2_j$ and $\phi^2_j$ such that $\sum_j \int_{\R^3} \psi_j^2(x) \d x\leq 1$ and $\sum_j \int_{\R^d} \phi_j^2(y) \d y\leq 1$.

A similar statement with \lq $\geq$ \rq  holds with a $\liminf_{t\to\infty}$ and an open set $G\subset\widetilde{\mathcal X}^\otimes$. Also, 
for the total mass $Z_t$ in \eqref{rewrite}, we can take $G=F=\widetilde{\mathcal X}^\otimes$ in the two bounds proved above and conclude
\begin{equation}\label{totalmass}
\begin{aligned}
\lim_{t\to\infty}\frac 1t \log Z_t
&= \sup_{\xi^\otimes\in\widetilde{\mathcal X}^\otimes} \sum_j \bigg\{\int_{\R^d}\int_{\R^d} V_\delta(x-y)\psi_j^2(x)\phi_j^2(y) \d x\d y\\
&\qquad\qquad\qquad\qquad-\frac{1}{2}\sum_j\big\|\nabla \psi_j\big\|_2^2-\frac{1}{2}\sum_j\big\|\nabla \phi_j\big\|_2^2\bigg\}
\end{aligned}\end{equation}

Let us now turn to the singular potential $V(x)=\frac 1 {|x|}$. Note that, $V(x)-V_\delta(x)\leq \frac{C\sqrt\delta}{|x|^{3/2}}$. We can estimate by H\"older's inequality,
$$
\E^\otimes\big[\exp\{t H(\widetilde L_t^\otimes)\}\big] \leq \E^\otimes\big[\exp\{p\, t H(\widetilde L_t^\otimes)\}\big]^{1/p}\,\, \E^\otimes\big[\exp\{q\,t (H-H_\delta)(\widetilde L_t^\otimes)\}\big]^{1/q}
$$
and use the large deviation estimates obtained before for any fixed $\delta>0$. Then if we let $p\to 1$ followed by $\delta\to 0$, then we obtain the required LDP for the distributions $\mathbb S_t$
with rate function $J(\cdot)$ defined in \eqref{Jdef}, provided we show that
$$
\limsup_{\delta\to 0}\limsup_{t\to\infty}\frac{1}{t}\log \E^\otimes\bigg[ \exp \bigg\{\frac{C\sqrt\delta}{t}\int_0^t\int_0^t \frac 1{|W^{\ssup 1}_s-W^{\ssup 2}_\sigma|}\d s\d\sigma\bigg\}\bigg]=0
$$ 
for any $C>0$. The above estimate is a routine check (e.g., Lemma 3.7, \cite{DV83-P}) and its proof is omitted. 

In particular, if $U$ denotes any open neighborhood of $\widetilde{\mu_0\mu_0}$ in the space $\widetilde{\mathcal X}^\otimes$, then, 
\begin{equation}\label{rhotilde1}
\begin{aligned}
&\limsup_{t\to\infty}\frac 1t \log \mathbb S_t(U^c)\\
&\leq -\inf_{\xi^\otimes=\{\widetilde{\alpha_j\beta_j}\}_j\notin U} \,\,\bigg[\widetilde\rho- \sum_j \bigg\{\int \frac{ \alpha_j(\d x)\beta_j(\d y)}{|x-y|}-I({\alpha}_j)- I(\beta_j)\bigg\}\bigg].
\end{aligned}
\end{equation}
Since the function $\frac 1 {|x|}$ is positive definite, for any measures $\alpha$ and $\beta$,
$$
2\int\int \frac{\psi^2(x)\phi^2( y)}{|x-y|} \, \d x \d y\leq \int\int \frac{\psi^2(x)\psi^2(y)}{|x-y|} \, \d x \d y+\int\int \frac{\phi^2(x)\phi^2(y)}{|x-y|}\,\, \d x \d y.
$$
Hence, the variational formula \eqref{rhotilde} reduces to,
$$
\widetilde\rho=\sup_{(\widetilde\alpha_j)_j\in \widetilde{\mathcal X}^\otimes} \,\,\sum_j \bigg\{\int_{\R^3}\int_{\R^3} \frac {\psi_j^2(x)\psi_j^2(y)} {|x-y|} \d x\d y-\big\|\nabla \psi_j\big\|_2^2\bigg\}
$$
so that $\alpha_j$ has density $\psi_j^2$ and $\sum_j\int\psi_j^2(x)\d x{\leq 1}$.
Now we can invoke the same rescaling argument as in the proof of Theorem \ref{thm2-LDPBrowInt}, to see that the above
supremum is attained when $\alpha_j$ consists of only one function $\psi^2(x)\d x$  with $\|\psi\|_2=1$. Hence, $\widetilde\rho=\rho$, 
and we recall that the variational problem \eqref{rhodef} for $\rho$ is attained at a unique radially symmetric function $\psi_0$. Since, $\psi_0^2$ is the density of $\mu_0$,
the infimum on the right hand side of \eqref{rhotilde1} must be strictly positive. This finishes the proof of Theorem \ref{thmpairtube}.
\end{proof}

 \section{Lyapunov exponents and intermittency of the parabolic Anderson problem in $\R^d$.}\label{sec-PAM}
 
Let us now consider the stochastic partial differential equation written formally as 
\begin{equation}\label{pamdef}
\partial_t Z= \frac 12 \Delta Z+  Z \eta ,
\end{equation}
with a prescribed initial condition. Here $\eta$ denotes white noise in $\R^d$, which is a centered Gaussian process with covariance kernel $\E\big(\eta(x) \eta(y)\big)= \delta_0(x-y)$. 
Note that \eqref{pamdef} is only a formal expression, and in attempting to give any precise meaning to it, one is immediately faced with a problem of multiplying distributions. 
Since the random field $\eta=\{\eta(x)\}_{x\in\R^d}$ can not be defined pointwise and the product $Z \eta$ is ill-defined, one needs a smoothing procedure leading to a mollified
and well defined version of \eqref{pamdef}, and much recent progress has been made (\cite{HL15}) in proving that such smoothened solutions converge
(after a suitable renormalization procedure) to a well-defined solution to the equation \eqref{pamdef} in $\R^3$, based on the theory of regularly structures (\cite{H14}).

It is typical of the smoothened solution of the equation \eqref{pamdef} to be ``intermittent", as one expects the moments of these solutions to have exponential growth 
(as the smoothing parameter vanishes), though the supports of the solutions  are supported on thin sets in $\R^d$. Thus, for small smoothing parameter, 
these solutions are distinguished by formation of a peculiar spatial structure of strong pronounced ``islands" (such as sharp peaks) which determine the main contribution 
to the physical process in such media, see Remark \ref{rmk-intermittency}.
The large deviation theory developed in Section \ref{sec-paircompact} allows a direct computation of the asymptotic growth rate of all 
moments of the smoothened solution under a suitable rescaling, and the growth rates are quantified by explicit variational formulas. 
Let us now turn to a formal definition of the model and a precise statement of the main result.

Let us fix $d\geq 3$ and for each $\eps>0$, let $\varphi_\eps$ be a smooth mollifier in $\R^d$, i.e.,  $\varphi_\eps(x)= \eps^{-d} \varphi(x/\eps)$ for some smooth, positive definite, even function $\varphi$ with compact support and $\int_{\R^d} \varphi=1$.  Then $\int_{\R^d} \varphi_\eps=1$
and $\varphi_\eps\Rightarrow \delta_0$.

Let $\mathcal S= \mathcal S(\R^d)$ denote the Schwartz space of rapidly decreasing functions, while 
$(\Omega,\mathcal F, \P)$ denotes a complete probability space. We also 
denote by $\eta=\{\eta(f)\}_{f\in \mathcal S}$ a centered Gaussian field with covariance
$\E\{\eta(f) \eta(g)\}= \int_{\R^d} f(x) g(x) \d x$. Such a field can also be defined pointwise in $\R^d$ as
$$
\eta_\eps(x)= \eta\big(\varphi_\eps(x-\cdot)\big)= \big(\eta\star \varphi_\eps\big)(x).
$$
Note that $\eta_\eps=\{\eta_\eps(x)\}_{x\in\R^d}$ is also a centered Gaussian process with covariance 
\begin{equation}\label{cov}
\E\big\{\eta_\eps(x)\eta_\eps(y)\big\}= \int_{\R^d} \varphi_\eps(x-z) \varphi_\eps(y-z) \, \d z= \big(\varphi_\eps\star \varphi_\eps\big)(x-y)
= V_\eps(x-y),
\end{equation}
and we denoted $V_\eps=\varphi_\eps\star \varphi_\eps$. The mollified and rescaled equation corresponding to \eqref{pamdef} is given by
\begin{equation}\label{pameps}
\begin{aligned}
&\partial_t Z_\eps= \frac 12 \Delta Z_\eps + C(\eps) \, Z_\eps \eta_\eps\\
&Z_\eps(0,x)= 1
\end{aligned}
\end{equation}
where
\begin{equation}\label{rescaling}
C(\eps)= \eps^{\frac{d-2}2} \qquad d\geq 3.
\end{equation}
We would like to study the 
asymptotic growth rate of the moments of its Feynman-Kac solution 
\begin{equation}\label{FK}
Z_\eps(t,x)= E_x \bigg\{\exp\bigg\{C(\eps)\int_0^t \eta_\eps(W_s) \, \d s\bigg\} \, \bigg\}.
\end{equation}
as $\eps\to 0$. In the above expression, $E_x$ refers to the expectation with respect to the Wiener measure $P_{x}$ for a Brownian motion starting at $x\in\R^d$.
Since we are interested in the behavior of $Z_\eps(t,x)$ as $\eps\to 0$ for fixed $t$, we will write $Z_\eps(x)=Z_\eps(1,x)$ and study
the asymptotic behavior of
$$
m_p(\eps,x)= \E\big[Z_\eps(x)^p\big]
$$
for all $p=1,2,3,\dots$. The following result, whose proof is based on a simple application of Theorem \ref{thmldp}, determines explicit variational formulas for the
``annealed Lyapunov exponents" which determine the asymptotic behavior of $m_p(\eps,x)$ as $\eps\to 0$. 
\begin{theorem}\label{Lyapexp}
For any $p\in \N$ and $x\in \R^d$ with $d\geq 3$, 
\begin{equation}\label{mdef}
\begin{aligned}
\hspace{10mm}\lim_{\eps\to 0} \eps^2 &\log m_p(\eps,x)
= m_p\\
&= 2^{p-1} \sup_{\heap{\psi_j\in H^1(\R^d)}{\sum_j\|\psi_j\|_2\leq1}} \, \sum_j \bigg\{ 2^{p-2}\int\int_{\R^d\times\R^d} V(x-y) \psi_j^2(x) \psi_j^2(y) \d x \d y \\
&\qquad\qquad\qquad\qquad\qquad- \frac 12 \|\nabla \psi_j\|_2^2\bigg\},
\end{aligned}
\end{equation}
where $V=V_\varphi=\varphi \star \varphi$. 
\end{theorem}
\begin{proof} We fix any starting point $x\in \R^d$ and handle the case $p=1$ first. Then, 
$$
\begin{aligned}
m_1(\eps,x)=\E\big(Z_\eps(x)\big)&=\E\bigg\{E^{\ssup x}\bigg(\e^{\eps^{\frac d2 - 1} \int_0^1 \eta_\eps(W_s) \d s }\bigg)\bigg\}\\
&= E^{\ssup x}\bigg[\exp\bigg\{\frac 12\eps^{d-2}\, \int_0^1\int_0^1 \d\sigma\d s \,V_\eps(W_\sigma- W_s\big)\bigg\} \bigg],
\end{aligned}
$$
since $\{\eta_\eps(x)\}_{x\R^d}$ is centered Gaussian with covariance given by \eqref{cov}. But,
$$
V_\eps(x-y)=\big(\varphi_\eps\star \varphi_\eps\big)(x-y)= \eps^{-d}V\bigg(\frac{x-y}\eps\bigg),
$$ 
where $V=\varphi\star\varphi$. By Brownian scaling,
$$
\begin{aligned}
m_1(\eps,x)&= E_{x}\bigg[\exp\bigg\{\frac 12\eps^{d-2}\, \eps^{-d}\, \int_0^1\int_0^1 \d\sigma\d s \,V(\eps^{-1}(W_\sigma- W_s))\bigg\}\bigg] \\
&= E_{x}\bigg[\exp\bigg\{\frac 12\eps^{d-2}\, \eps^{-d}\, \int_0^1\int_0^1 \d\sigma\d s \,V\big(W_{\sigma/\eps^2}- W_{s/\eps^2}\big)\bigg\}\bigg]\\
&=E_{x} \bigg[\exp\bigg\{\frac 12\eps^{2}\,  \int_0^{1/\eps^2}\int_0^{1/\eps^2} \d\sigma\d s \,V\big(W_\sigma- W_s\big)\bigg\}\bigg] \\
&= E_{x} \bigg\{\exp\bigg\{\frac 1 {2\tau} \int_0^{\tau}\int_0^{\tau} \d \sigma \d s \, V(W_\sigma- W_s\big)\bigg\}\bigg\}
\end{aligned}
$$
for $\tau=\eps^{-2}$. Note that, the last expression can also be written as $E_x[\exp\{\tau H(\widetilde L_\tau)\}]$ where,
for the probability measure $L_\tau=1/\tau\int_0^\tau \d s \delta_{W_s}\in \Mcal_1$ and $H(\widetilde L_\tau)=\int\int V(x-y)L_\tau(\d x)L_\tau(\d y)$ 
is a continuous functional on the compact metric space $\widetilde{\mathcal X}^\otimes$ (Corollary \ref{corcont}). By a full large deviation principle for the distribution 
of $\widetilde L_\tau$ (Theorem \ref{thmldp}) and Varadhan's lemma applied to the functional $H$, we have
$$
\begin{aligned}
&\lim_{\eps\to 0}\eps^2 \log m_1(\eps,x)\\
&=
\sup_{\sum_j\|\psi_j\|_2\leq1} \sum_j \bigg\{\frac 12\int\int_{\R^d\times\R^d} V(x-y) \psi_j^2(x) \psi_j^2(y) \d x \d y - \frac 12 \|\nabla \psi_j\|_2^2\bigg\}
=m_1,
\end{aligned}
$$
proving Theorem \ref{Lyapexp} for $p=1$.

Let us now turn to the case $p\geq 2$ and focus on the case $p=2$ for simplicity. Let $E^{\otimes}_x$ denote the joint distribution of two independent Brownian motions
$W^{\ssup 1}, W^{\ssup 2}$, both starting at $x\in \R^d$. Then, using similar scaling relations as before,
$$
\begin{aligned}
m_2(\eps,x)&= \E\bigg[ E_x^{\otimes} \bigg\{ \e^{\sum_{i=1}^2 \eps^{d/2 -1}\int_0^1 \eta_\eps\big(W^{\ssup i}_s\big) \, \d s }\bigg\}\bigg]\\
&= E_x^{\otimes} \bigg[\exp\bigg\{ \frac 12 \sum_{i,j=1}^2 \eps^{2}\int_0^{1/\eps^2} \int_0^{1/\eps^2} V\big(W_\sigma^{\ssup i}- W_s^{\ssup j}\big) \d\sigma\d s \bigg\}\bigg]\\
&= E_x^{\otimes} \bigg[\exp\bigg\{\frac\tau 2 H(\widetilde L_\tau^\otimes)\bigg\}\bigg].
\end{aligned}
$$
with $\tau=\eps^{-2}$ and $L_\tau=L^{\ssup 1}_\tau\otimes L_\tau^{\ssup 2}$. Then, again by Theorem \ref{thmldp} and Corollary \ref{corcont}, 
$$
\begin{aligned}
&\lim_{\eps\to 0}\eps^2 \log m_2(\eps,x)\\
&=
\sup_{\heap{\sum_j\|\psi_j\|_2\leq 1}{ \sum_j\|\phi_j\|_2\leq1}} \,\, \sum_j \bigg\{\frac 12\int\int_{\R^d\times\R^d} \d x \d y \,V(x-y)\psi_j^2(x)\phi_j^2(y) - \frac 12\|\nabla \psi_j\|_2^2- \frac 12\|\nabla \phi_j\|_2^2\bigg\}
\end{aligned}
$$
Since $V=\varphi\star\varphi$ is positive definite, for any $j$,
$$
\begin{aligned}
&\frac 12\int\int_{\R^d\times\R^d} \d x \d y \,V(x-y)\psi_j^2(x)\phi_j^2(y) \\
&\leq \int\int_{\R^d\times\R^d} \d x \d y \,V(x-y)\psi_j^2(x)\psi_j^2(y)+ \int\int_{\R^d\times\R^d} \d x \d y \,V(x-y)\phi_j^2(x)\phi_j^2(y),
\end{aligned}
$$
and the last variational formula reduces to
$$
2\sup_{\sum_j\|\psi\|_j\leq 1} \,\, \sum_j\bigg\{\int\int_{\R^d\times\R^d} \d x \d y \,V(x-y)\psi_j^2(x)\psi_j^2(y) -  \frac 12\|\nabla \psi_j\|_2^2\bigg\} =m_2.
$$
\end{proof}
\begin{lemma}\label{pam-lemma}
If $V$ satisfies $V(x/\sigma) \geq \sigma V(x)$ for any $\sigma\in (0,1]$, then, for any $p\in \N$,
$$
m_p=2^{p-1} \sup_{\heap{\psi_\in H^1(\R^d)}{\|\psi\|_2=1}} \,  \bigg\{ 2^{p-2}\int\int_{\R^d\times\R^d} V(x-y) \psi^2(x) \psi^2(y) \d x \d y - \frac 12 \|\nabla \psi\|_2^2\bigg\}
$$
\end{lemma}
\begin{proof}
The proof is similar to the argument appearing in Theorem \ref{thm2-LDPBrowInt}. Indeed, let us set
$$
\rho(\sigma)= \sup_{ \int_{\R^d} \psi^2(x) \d x  =\sigma} \bigg\{\int\int V(x-y) \psi^2(x)\psi^2(y) \d x \d y- \|\nabla\psi\|_2^2\bigg\},
$$
and for any $\psi$ with $\int_{\R^d} \psi^2(x)\d x=1$ we rescale by with $\psi_\sigma(x)= \sigma^{(1+d)/2} \psi(\sigma x)$. Then the integral $\int_{\R^d} \psi_\sigma^2(x) =\sigma$, while 
$$
\begin{aligned}
&\int\int V(x-y) \psi_\sigma^2(x)\psi_\sigma^2(y) \d x \d y- \|\nabla\psi_\sigma\|_2^2\\
&= \sigma^2 \int\int V\bigg(\frac{x-y}\sigma\bigg) \psi^2(x)\psi^2(y) \d x \d y-  \sigma^3 \|\nabla\psi\|_2^2 \\
&\geq \sigma^3\bigg[ \int\int V(x-y) \psi^2(x)\psi^2(y) \d x \d y-   \|\nabla\psi\|_2^2\bigg].
\end{aligned}
$$
Hence, $\rho(\sigma) \geq C \sigma^3$, implying that $\rho(\sigma_1+\sigma_2) > \rho(\sigma_1)+\rho(\sigma_2)$.
Hence, in \eqref{mdef}, the supremum must attain at a single function $\psi\in H^1(\R^d)$ 
such that $\|\psi\|_2=1$.
\end{proof}

\begin{remark}\label{rmk-intermittency}
Note that it follows from the above computations that 
$$
m_1<\frac {m_2}2 <\frac{m_3}3<\cdots.
$$
If we now choose $\theta_1$ with $m_1<\theta_1<\frac {m_2}2$, then by Theorem \ref{Lyapexp},
$$
\P[Z_\eps(x)> \e^{\theta_1\eps^{-2}}] \leq \e^{-\theta_1\,\eps^{-2}} \,\E[Z_\eps(x)] \approx \e^{-\eps^{-2}(\theta_1- m_1)}.
$$
Since the process $Z_\eps(x)$ is translation-invariant with respect to $x$, the spatial ergodic theorem then implies 
$$
\P[Z_\eps(x)> \e^{\theta_1\eps^{-2}}]= \lim_{r\to\infty} \frac 1 {|B_r|} \int_{B_r} \d x \,\, \1\big\{x\in\R^d\colon Z_\eps(x) > \e^{\theta_1\eps^{-2}}\big\}.
$$
Hence, as $\eps\to 0$, the asymptotic density of the sets $\big\{x\colon Z_\eps(x) > \e^{\theta_1\eps^{-2}}\big\}$ is exponentially small. On the other hand, 
$$
\E\big[Z_\eps^2(x)\big] \leq \E\big[Z_\eps^2(x) \1\big\{ Z_\eps(x) > \e^{\theta_1\eps^{-2}}\big\}\big] + \e^{2\theta_1\eps^{-2}}.
$$
Since $\E\big[Z_\eps^2(x)\big]\approx\e^{m_2\eps^{-2}}$ and $m_2>2\theta_1$, the last display implies that the essential contribution 
to the second moment of $Z_\eps$ comes from the set $\big\{ Z_\eps(x) > \e^{\theta_1\eps^{-2}}\big\}$ whose spatial density is exponentially small. 
Continuing recursively, one sees that although the moments of $Z_\eps$ grow rapidly, the concentration of mass of the physical process comes from 
a peculiar spatial structure, consisting of sharply peaked islands
of small diameter, known as {\it{intermittency}} (see \cite{GM90}).
\end{remark}

{\bf{Acknowledgement.}} The author wishes to thank S. R. S. Varadhan for reading an early draft of the manuscript and many valuable suggestions.

        








\frenchspacing
\bibliographystyle{plain}



                                



\end{document}